\documentclass[12pt]{article}

\usepackage{epsfig}
\usepackage{amsmath,amsthm}
\usepackage{amssymb,amsfonts}
\usepackage[french,english]{babel}
\usepackage{MnSymbol}
\usepackage[latin1]{inputenc}
\usepackage{color}
\usepackage{mathrsfs}
\usepackage{hyperref}
\usepackage[usenames,dvipsnames,svgnames,table]{xcolor}

\hypersetup{linktocpage}

\newtheorem{theorem}{Theorem}[section]
\newtheorem{proposition}[theorem]{Proposition}
\newtheorem{corollary}[theorem]{Corollary}
\newtheorem{lemma}[theorem]{Lemma}

\theoremstyle{definition}

\newtheorem{remark}[theorem]{Remark}
\newtheorem{definition}[theorem]{Definition}

\def\[#1\]{\begin{align*}#1\end{align*}}

\def\hM{\hat{M}}

\def\hx{\hat{x}}
\def\hy{\hat{y}}

\def\hX{\hat{X}}

\def\hZ{\hat{Z}}

\def\hf{\hat{f}}

\def\dr{\mathscr{D}_{R}}
\def\Dr{\Delta_{R}}
\def\odr{\mathcal{O}_{\mathscr{D}_{R}}}
\def\oDr{\mathcal{O}_{\Delta_{R}}}
\def\lns{\mathscr{L}_{NS}}
\def\lr{\mathscr{L}_{R}}
\def\pq{\pi_Q}

\def\rarrow{\rightarrow}

\def\nablabar{\overline{\nabla}}
\def\hnabla{\hat{\nabla}}
\def\hgamma{\hat{\gamma}}

\def\sbar{\overline{S}}

\def\hS{\hat{S}}
\def\Oh{\mathcal{O}}
\def\onq{|_{q}}

\def\q{q=(x,\hat{x};A)}
\def\qz{q_0=(x_0,\hat{x}_0;A_0)}

\def\R{\mathbb{R}}
\def\I{\mathcal{I}}

\def\R{\mathbb{R}}

\def\G{\mathbb{G}}

\newcommand{\SE}{\mathrm{SE}}
\newcommand{\SO}{\mathrm{SO}}
\newcommand{\GL}{\mathrm{GL}}
\newcommand{\Aff}{\mathrm{Aff}}
\newcommand{\mc}[1]{\mathcal{#1}}

\newcommand{\ol}[1]{\overline{#1}}
\newcommand{\pa}[1]{\frac{\partial}{\partial{#1}}}
\newcommand{\n}[1]{\left\lVert#1\right\rVert}
\newcommand{\fII}{\mathrm{II}}

\begin{document}

\title{Horizontal Holonomy for Affine Manifolds
\thanks{This research was partially supported by the iCODE institute, research project of the Idex Paris-Saclay.}
}

\author{Boutheina HAFASSA\thanks{bhafassa@gmail.com, Universit\'e Paris-Sud 11, STITS, and, Universit\'e de Tunis El Manar} \and Amina MORTADA\thanks{amina.mortada@gmail.com, Universit\'e Paris Sud 11 \& Universit\'e Libanaise, LNCSR Scholar} \and Yacine CHITOUR\thanks{yacine.chitour@lss.supelec.fr, Universit\'e Paris-Sud 11, Sup\'elec, Gif-sur-Yvette, 91192, France} \and Petri KOKKONEN\thanks{pvkokkon@gmail.com}
}

\date{\today}

\maketitle

\begin{abstract}
In this paper, we consider a smooth connected finite-dimensional manifold $M$, an affine connection $\nabla$ with holonomy group
$H^{\nabla}$ and $\Delta$ a smooth completely non integrable distribution. We define the $\Delta$-horizontal holonomy group $H^{\;\nabla}_\Delta$ as the subgroup of $H^{\nabla}$ obtained by $\nabla$-parallel transporting frames only along loops tangent to $\Delta$. We first set elementary properties of $H^{\;\nabla}_\Delta$ and show how to study it using the rolling formalism (\cite{ChitourKokkonen}).
In particular, it is shown that $H^{\;\nabla}_\Delta$ is a Lie group. 
Moreover, we study an explicit example where $M$ is a free step-two homogeneous Carnot group and $\nabla$ is the Levi-Civita connection associated to a Riemannian metric on $M$,
and show that in this particular case the connected component of the identity of $H^{\;\nabla}_\Delta$ is compact and strictly included in $H^{\nabla}$.
\end{abstract}


\section{Introduction}
The purpose of this paper consists of extending the concept of horizontal holonomy of an affine connection in the context of distributions on a manifold
i.e., subbundles of the tangent bundle of a manifold. 
More precisely, consider the triple $(M,\nabla,\Delta)$ where $M$ is an $n$-dimensional smooth connected manifold, $\nabla$ is an affine connection on $M$ (one says then that  $(M,\nabla)$ is an affine manifold) and  $\Delta$ is a smooth distribution on $M$. 
One furthermore assumes that $\Delta$ is completely controllable, i.e., every pair of points in $M$ can be connected by an absolutely continuous tangent to the distribution $\Delta$.
Recall that the holonomy group $H^\nabla$  of $\nabla$ as the subgroup of $\GL(n)$ obtained (up to conjugation) by $\nabla$-parallel transporting frames along absolutely continuous (or piecewise smooth) loops of $M$.

For every point $x\in M$, we define the subset $H^\nabla_\Delta |_{x}$ of 
$H^\nabla |_{x}$,
the holonomy group of $\nabla$ at $x$,
obtained by parallel transporting, with respect to $\nabla$, frames of $M$ along a restricted set of absolutely continuous $\Delta$-horizontal loops based at $x$,
namely along loops which are tangent (almost everywhere) to the distribution $\Delta$.
Thanks to the hypotheses of connectedness of $M$ and complete controllability of $\Delta$, one can deduce that the sets $H^\nabla_\Delta |_{x}$, $x\in M$ are all conjugate to a Lie-subgroup  $H^\nabla_\Delta$ of  $H^\nabla$ that we call $\Delta$-horizontal (or simply horizontal) holonomy group of $\nabla$. In the case where $\nabla$ is the Levi-Civita connection associated to some Riemannian metric $g$ on $M$, one can take both $H^\nabla_\Delta$ and $H^\nabla$ as subgroups of $\mathrm{O}(n)$ and even $\SO(n)$ if in addition $M$ is assumed to be oriented. 
Understanding the relationships between  $H^\nabla$ and $H^\nabla_\Delta$ appears to be an interesting challenge. For instance, given an affine manifold $(M,\nabla)$, determining necessary and (or) sufficient conditions on a completely controllable distribution $\Delta$ of $M$
so that the $H^\nabla_\Delta$ equals $H^\nabla$
is not an obvious question, besides trivial cases.
Another issue to be addressed consists of fixing the pair manifold and distribution i.e., $(M,\Delta)$ and then make the connection $\nabla$ vary. One question could be to undestand if there are connections ''more adapted or intrinsic'' than others (in a sense to be defined) for the pair  $(M,\Delta)$. Moreover, one could also study the mapping $g\mapsto H^{\nabla^g}_\Delta$ where $g$ is a complete Riemannian metric on $M$ and $\nabla^g$ the corresponding Levi-Civita connection, 
for instance describing the range of this mapping. Note that such issues have been already addressed in  \cite{FalbelGorodskyRumin} where the authors consider the case of manifolds of contact type with a distribution arising from an adapted connection.

In this paper, we essentially start this program by defining precisely the $\Delta$-horizontal holonomy group associated to a given admissible triple $(M,\nabla,\Delta)$. Our first main result besides elementary ones is the following: we prove that if $\Delta$ is a constant rank completely controllable distribution, then  $H^\nabla_\Delta$ is a connected Lie subgroup of $\GL(n)$ (or $\mathrm{O}(n)$ if $\nabla$ is the Levi-Civita connection of some Riemannian metric on $M$). This enables us to study $H^\nabla_\Delta$ via its differentiable structure. Moreover, we also propose to study $\Delta$-horizontal holonomy groups by recasting them within the framework of rolling manifolds.
Indeed, recall that E. Cartan defines holonomy groups in \cite{Cartan} as what is called now affine holonomy group by ``developing'' a manifold its tangent space at any point. This procedure has been generalized in \cite{ChitourKokkonen,Godoyandelse,Sharpe} to an arbitrary pair of Riemannian manifolds of same dimension and it is also called as ``rolling a Riemannian manifold onto another one without slipping nor spinning''. Yet, that type of rolling was extended in \cite{MortadaKokkonenChitour} to the case where both manifolds do not have necessarily the same dimension. See also \cite{ChitourGodoyKokkonen} for a historical account as well as applications of the rolling of manifolds. 

In the present situation, the rolling framework amounts to define an $n$-dimensional smooth distribution $\dr$, called the rolling distribution, on the  state space $Q$ defined as the fiber bundle over the product of $(M,\nabla)$ and $(\mathbb{R}^n,\hnabla^n)$ where $\hnabla^n$ is the Euclidean connection on $\R^n$ and the typical fiber over $(x,\hat{x})\in M\times \mathbb{R}^n$ is identified with the set of endomorphisms of $T_xM$. For every $q\in Q$, let $\Oh_{\dr}(q)$, be the $\dr$-orbit through $q$, i.e., the set of endpoints of the absolutely continuous curves starting at $q$ and tangent to $\dr$. Then, for every $q\in Q$
and $x'\in M$, the fiber of $\Oh_{\dr}(q)$ over $x'$ (if non-empty)
is conjugate to a subgroup of $\mathbb{R}^n \rtimes \GL(n)$ whose $\GL(n)$-part is exactly
$H^\nabla$. Moreover, since  $\Oh_{\dr}(q)$ is an immersed manifold in $Q$ whose tangent space at every $q'\in \Oh_{\dr}(q)$ contains the (evaluation at $q'$ of the) Lie algebra generated by vector fields tangent to $\dr$ (cf. \cite{AgrachevSachkov,Jean,Jurdjevic}), it is possible to determine elements of the Lie algebra of $H^\nabla$ as Lie brackets of vector fields tangent to $\dr$. Given now a completely controllable distribution $\Delta$, one can define a subdistribution $\Dr$ of $\dr$ on $Q$ so that, for every $q\in Q$ and
$x'\in M$ the fiber over $x'$ of $\Oh_{\Dr}(q)$ is conjugate to a subgroup of $\mathbb{R}^n \rtimes \GL(n)$ whose $\GL(n)$-part is now equal to $H^\nabla_\Delta$. Since the latter has been proved to be a Lie group, one can determine elements of its Lie algebra by computing Lie brackets of vector fields tangent to $\Dr$. Note that, as also mentioned above, $\GL(n)$ can be replaced by $\mathrm{O}(n)$ ($\SO(n)$ respectively) if $\nabla$ is the Levi-Civita connection of some Riemannian metric on $M$ (if in addition $M$ is oriented). 

We use that approach to provide our second main result, namely an explicit example for
a strict inequality in $\dim(H^\nabla_\Delta)\leq \dim(H^\nabla)$. More precisely, we consider the triple $(M,\nabla,\Delta)$ where $M$ is a free step-two homogeneous Carnot group of $m\geq 2$ generators $(X_i)_{1\leq i\leq m}$, $\nabla$ is the Levi-Civita associated  with the Riemannian metric on $M$ defined in such a way that the $X_i$'s, $1\leq i\leq m$ and the Lie brackets $[X_i,X_j]$, $1\leq i<j\leq m$ form an orthonormal basis and $\Delta$ is the distribution defined by the span of the $X_i$'s, $1\leq i\leq m$. In this case $M$ is of dimension $m+n$ with $n=m(m-1)/2$. Then we prove that $(M,\nabla)$ has full holonomy i.e., $H^\nabla=\SO(n)$,
and that the connected component of the identity of $H^\nabla_\Delta$
is a closed Lie subgroup of $\SO(n)$ of dimension $m+n$.

We close this introduction by describing the structure of the paper. We gather in the second section most of the required notations and we precisely define the $(\nabla,\Delta)$-holonomy group first using classical concepts and secondly by relying on the rolling framework. In the fourth section, we consider in details the example of the free step-two homogeneous Carnot group of $m\geq 2$ generators and we conclude with an appendix containing a technical result needed in the third section.

{\bf Acknowledgement.} The authors thank F. Jean and M. Sigalotti for fruitful discussions and insights regarding the proof of Proposition~\ref{regular-0}.


\section{Notations}\label{Notations}

Let $M$ be an $n$-dimensional smooth connected manifold where $n$ is a positive integer.
Let $\mathcal{X} (M)$ be the set of smooth vector fields on $M$. An affine connection $\nabla$ on $M$ is a $\R$-bilinear map
\[
\mathcal{X} (M) \times \mathcal{X} (M)\rarrow \mathcal{X} (M); \quad  (X , Y) \mapsto \nabla_X Y,
\]
which is $C^{\infty} (M)$-linear in the first variable and verifies the Leibniz rule over $C^{\infty} (M)$ in the second variable. The pair $(M,\nabla)$ is said to be an affine manifold. If, moreover, the exponential map $\exp^{\nabla}_x$ of $(M,\nabla)$ is defined on the whole tangent space $T_x M$ for all $x \in M$, then $(M,\nabla)$ is said to be a (geodesically) complete affine manifold. We use $\nabla^n$ and  $\nabla^g$ respectively to denote the Euclidean connection on $\R^n$ and  the Levi-Civita connection of a Riemannian manifold $(M,g)$.  The notation $[\cdot,\cdot]$ stands for the Lie bracket operation in $TM$.

We define the curvature tensor $R^{\nabla}$ and the torsion tensor $T^{\nabla}$ of a affine connection $\nabla$ as
\[
R^{\nabla} (X,Y) Z &= \nabla_X \nabla_Y Z -  \nabla_Y \nabla_X Z - \nabla_{[X,Y]} Z,\\
T^{\nabla} (X,Y) &= \nabla_X Y - \nabla_Y X - [X,Y],
\]
respectively, for smooth vector fields $X,Y,Z$ on $M$.

If $\gamma:I \rarrow M$ is any absolutely continuous (a.c. for short) curve in $M$ defined on a real interval $I$ containing $0$, we use $(P^{\nabla})_0^{t} (\gamma) T_0$, $t\in I$, to denote
the $\nabla$-parallel transport along $\gamma$ of a tensor $T_0$ of rank $(m,k)$ at $\gamma(0)$. It 
is the unique solution for $T(t)$ (in terms of tensor fields of rank $(m,k)$ defined along $\gamma$)
to the Cauchy problem
\[
\nabla_{\dot{\gamma} (t)} T(t) = 0, \quad \text{for a.e. } \; t \in I,\
T(0) = T_0.
\]
Let $(\hM, \hnabla)$ be another affine manifold and $f: M \rarrow \hM$ be a smooth map. we say that $f$ is affine if for any a.c. curve $\gamma:[0,1] \rarrow M$, one has
\begin{equation}\label{e3.32}
f_\star |_{\gamma (1)} \circ {(P^\nabla)}_0^1 (\gamma) = {(P^{\hnabla})}_0^1 (f \circ \gamma) \circ f_\star  |_{\gamma (0)} .
\end{equation}

An a.c. curve $\gamma:[a,b] \rarrow M$ is a loop based at $x \in M$ if $\gamma (a) = \gamma (b) = x$. We denote by $\Omega_M (x)$ the space of all a.c. loops $[0,1] \rarrow M$ based at some given point $x \in M$. Moreover, if $\gamma: [0,1] \rarrow M$ and $\delta: [0,1] \rarrow M$ are two a.c. curves on $M$ such that $\gamma (0) = x$, $\gamma (1) = \delta (0) = y$ and $\delta (1) = z$ where $x,\;y,\;z \in M$, the concatenation $\delta \cdot \gamma$ is the a.c. curve defined by
\begin{align}\label{he1}
\delta \cdot \gamma : [0, 1] \rarrow M,\quad (\delta \cdot \gamma) (t) =  \left\lbrace
                         \begin{array}{ll}
                         \gamma (2t)   & \quad  t \in [0,\frac{1}{2}] ,\\
                         \delta(2t-1) & \quad t \in [\frac{1}{2},1].
                         \end{array}
                         \right.
\end{align}

The previous definitions allow us to state the subsequent definition of holonomy group.

\begin{definition}\label{d3.2}
For every $x\in M$, the holonomy group $H^{\nabla} |_{x}$  at $x$ is defined by
$$
H^{\nabla} |_{x} = \{ (P^{\nabla})_0^1 (\gamma) \mid \gamma \in \Omega_M (x)\}.
$$
\end{definition}
For every $x \in M$, $H^{\nabla} |_{x}$ is a subgroup of $\GL (T_{x} M)$, the group of isomorphisms of $T_{x} M$, which is  clearly
isomorphic to $\GL(n)$ the group of $n \times n $ invertible matrices with real entries. Since $M$ is connected, it is well-known that, for any two points $x,y\in M$,
$H^{\nabla} |_{x}$ and $H^{\nabla} |_{y}$ are conjugate subgroups of $\GL (T_x M)$
and thus one can define $H^{\nabla}\subset \GL(n)$ the holonomy group of the affine connection $\nabla$ (cf. \cite{Joyce}).

\ \ \ \ \ We also recall that a
smooth distribution $\Delta$ on $M$ is a smooth subbundle of $TM$
 The flag of $\Delta$ is the collection of the distributions $ \Delta^j$, $j\geq 1$,
where, for every $x\in M$,
$\Delta^1 |_x := \Delta |_x$ and $\Delta^{s+1} |_x := \Delta^s |_x + [\Delta^1 , \Delta^s]|_x$ for $s \geq 1$. We say that the distribution $\Delta$ on $M$ is of constant rank $m\leq n$ if $\dim( \Delta |_x)=m$ for every $x\in M$ and verifies the Lie algebraic rank condition (LARC) if, for any $x\in M$, there exists an integer $r = r(x)$ such that $\Delta^r |_x  = T_x M$. The number $r(x)$ is called the step of $\Delta |_x$ (cf. \cite{Jean} for more details).

An a.c. curve $\gamma : I \rarrow M$, $I$ bounded interval in $\mathbb{R}$, is said to be $\Delta$-admissible, or $\Delta$-horizontal, if it is tangent to $\Delta$ a.e. on $I$, i.e., if for a.e.
$t \in I$, $\dot{\gamma} (t) \in \Delta |_{\gamma(t)}$. For $x_0 \in M$, the $\Delta$-orbit through $x_0$, denoted $\Oh_{\Delta} (x_0)$, is the set of
endpoints of the $\Delta$-admissible curves of $M$ starting at $x_0$, i.e.,
\begin{align*}
\Oh_{\Delta} (x_0) = \{ \gamma(1) \mid \gamma : [0,1] \rarrow M, \text{ a.c. $\Delta$-admissible curve}, \; \gamma(0) = x_0 \}.
\end{align*}
By the Orbit Theorem (cf.\cite{Jean}), it follows that $\Oh_{\Delta} (x_0)$ is an immersed smooth submanifold of $M$ containing $x_0$ so that the tangent space $T_y\Oh_{\Delta} (x_0)$ for every $y\in \Oh_{\Delta} (x_0)$ contains $Lie_y(\Delta)$, the evaluation at $y \in M$ of the Lie algebra $Lie(\Delta)$ generated by $\Delta$. Furthermore, if a smooth distribution $\Delta'$ on $M$ is a subdistribution of $\Delta$ (i.e., $\Delta' \subset \Delta$), then $\Oh_{\Delta'} (x_0) \subset \Oh_{\Delta} (x_0)$ for all $x_0 \in M$. A smooth distribution $\Delta$ is said to be completely controllable if, for every $x\in M$, $\Oh_{\Delta} (x)=M$ i.e. any two points of $M$ can be joined by an a.c. $\Delta$-admissible curve.  Recall that, the Lie Algebra Rank Condition (LARC), i.e. $Lie_x(\Delta) = T_x M$, is a sufficient condition for the complete controllability of $\Delta$ (cf. \cite{Jean})
when $M$ is connected, which is what we assume in this paper.

\section{Affine Holonomy Group of $(M,\nabla,\Delta)$}

\subsection{Definitions}

Consider the triple $(M,\nabla, \Delta)$ where $M$ is a smooth manifold, $\nabla$ a affine connection on $M$ and $\Delta$ a completely controllable smooth distribution on $M$. In this section, we will restrict Definition \ref{d3.2} to the $\Delta$-admissible curves on $ M$. To this end, we will define the set of all $\Delta$-admissible loop based at points of $M$.

\begin{definition}
We define $\Omega_{\Delta} (x)$ the set of all a.c. $\Delta$-admissible loops based at $x$, as
$$
\Omega_{\Delta} (x) : = \{ \gamma \mid \gamma : [a,b] \rarrow M \; a.c., \; \gamma (a) = \gamma (b) = x \text{ and $\dot{\gamma} (t) \in \Delta \big|_{\gamma (t)}$ a.e.}  \}
$$
\end{definition}

The following result is immediate from the definitions.

\begin{proposition}\label{p3.8}
The set $\Omega_{\Delta} (x)$ of all a.c. $\Delta$-admissible loop based at $x$ is not empty and is closed under the operation $"\cdot"$ given in \eqref{he1}.
\end{proposition}

We define the holonomy group associated with the distribution $\Delta$ as follows.

\begin{definition}
For every $x\in M$, the holonomy group associated with $\Delta$ at $x$ is defined as
\begin{equation*}
H_\Delta^{\;\nabla} |_{x} : = \{ (P^{\nabla})_0^1 (\gamma) \mid \gamma \in \Omega_{\Delta} (x)\}.
\end{equation*}
\end{definition}

\begin{proposition}
For every $x,y\in M$, $H_\Delta^{\;\nabla} |_{x} $ is a subgroup of $H^{\nabla} |_{x} $ and $H_\Delta^{\;\nabla} |_{x} $ is conjugate to $H_\Delta^{\;\nabla} |_{y}$.  One can thus define $H^{\nabla}_{\Delta}\subset H^{\nabla}\subset \GL(n)$ and we call it the $\Delta$-horizontal holonomy group associated with $\Delta$ and the affine connection $\nabla$.
\end{proposition}

\begin{proof}
Since $\Omega_{\Delta} (x)$ is a nonempty set for any $x \in M$, then $H_\Delta^{\;\nabla} |_{x} $ is 
also a nonempty subset of $H^{\nabla} |_{x} $. By Definitions 2.2.1 and 2.2.2 of \cite{Joyce}, the 
inverse map of $(P^{\nabla})_0^1 (\gamma)$ is $(P^{\nabla})_0^1 (\gamma^{-1})$ and $(P^{\nabla})_0^1 (\delta) \circ (P^{\nabla})_0^1 (\gamma)$ is equal to $(P^{\nabla})_0^1 (\delta \cdot \gamma)$, for 
any $\gamma: [0,1] \rarrow M$ and $\delta: [0,1] \rarrow M$ belonging to $ \Omega_{\Delta} (x)$. Thus, 
we get the first statement. Next, taking into account the fact that $\Delta$ is completely 
controllable, one deduces the rest of the proposition.
\end{proof}


\begin{remark}\label{rieman0}
If $g$ is a Riemannian metric on the smooth manifold $M$ and $\nabla^g$ is the Levi-Civita connection associated to $g$, then the holonomy group $H^{\nabla^g} |_{x}$ with $x \in M$ is a subgroup of $O (T_{x} M)$, the set of $g$-orthogonal transformations of $T_{x} M$. If, moreover, $M$ is oriented, one can easily prove that $H^{\nabla^g} |_{x}$ is a subgroup of $\SO(T_{x} M)$. One can then define the holonomy group of $\nabla^g$ as a subgroup of $O(n)$ ($SO(n)$ respectively) the group of orthogonal transformations of the euclidean $n$-dimensional space (the subgroup of $O(n)$ with determinant equal to one if $M$ is oriented respectively).
\end{remark}

\subsection{Holonomy groups associated with distributions using the framework of rolling manifolds}\label{generalrolling}

Let $M$ be a smooth $n$-dimensional manifold and $\nabla$ a connection on $M$.
Set $(\hM, \hnabla) := (\mathbb{R}^n,\hnabla^n)$ where $\hnabla^n$ is the Euclidean connection on $\R^n$. We associate to $(M,\nabla)$ the curvature tensor $R^\nabla$ and to the product manifold $ (M,\nabla) \times (\R^n, \hnabla^n) $ the affine connection $\nablabar$.

\subsubsection{Affine Holonomy Group of $M$}

We recall next basic definitions and results stated in \cite{ChitourKokkonen}  and \cite{Kokkonen}.

\begin{definition}
The state space of the development of $(M,\nabla)$ on $(\R^n,\hnabla^n)$ is
$$
Q:= Q(M,\mathbb{R}^n) =\{A \in T^*_x M \otimes \mathbb{R}^n \mid A \in GL (T_x M), \; x \in M\}.
$$
A point $q \in Q$ is written as $q=(x,\hx;A)$. Note that the word ``development'' can also be replaced by ``rolling''.
\end{definition}

\begin{definition}
Let $\gamma:[0,1] \rarrow M$ be an a.c. curve on $M$ starting at $\gamma (0) = x_0$. We define the development of $\gamma$ on $T_{x_0} M$ with respect to $\nabla$ as the a.c. curve $\Lambda^{\nabla}_{x_0} (\gamma) : [0,1] \rarrow T_{x_0} M$
$$
\Lambda^{\nabla}_{x_0} (\gamma) (t) = \int_0^t (P^{\nabla})_s^0 (\gamma) \dot{\gamma} (s) ds, \quad \quad t \in [0,1].
$$
\end{definition}

The following result can be found from \cite{KobayashiNomizu}.
\begin{proposition}
Let $\nabla$ be the Levi-Civita connection of a Riemannian metric $g$.
Then for any a.c. curve $\Gamma:[0,1]\to T_{x_0} M$
there exists a maximal $T=T(\Gamma)$ such that $0<T\leq 1$
and an a.c. curve $\gamma:[0,T]\to M$ satisfying
\[
\Lambda^{\nabla}_{x_0} (\gamma) (t)=\Gamma(t),\quad \forall t\in [0,T].
\]
Moreover, one can take $T=1$ for all such $\Gamma$s if and only if $(M,g)$ is complete.
\end{proposition}

By identification of $T^*_x M \otimes \mathbb{R}^n$ as the space of all $\R$-linear maps from the tangent space $T_x M$ at $x \in M$ onto the tangent space of $\R^n$ at $\hx \in \R^n$, one gets the following definitions. 

\begin{definition}
Let $(x_0,\hx_0) \in M \times \R^N$, $A_0 \in T^*_{x_0} M \otimes \mathbb{R}^n$ and an a.c. curve $\gamma: [0,1] \rarrow M$ such that $\gamma (0) = x_0$. We define the development of $\gamma$ onto $\R^n $ through $A_0$ with respect to $\nablabar$ as the a.c. curve $\Lambda^{\nablabar}_{A_0} (\gamma) : [0,T] \rarrow M$ given by
\[
\Lambda^{\nablabar}_{A_0} (\gamma) (t) : = (\Lambda^{\hnabla^n}_{\hx_0} )^{-1} (A_0 \circ \Lambda^{\nabla}_{x_0} (\gamma)) (t), \quad \quad t \in [0,T]
\]
with $T=T(\gamma)$ as in the previous definition.

We also define the relative parallel transport of $A_0$ along $ \gamma$ with respect to $\nablabar$ to be the linear map
\begin{align*}
&(P^{\nablabar})_0^t (\gamma) A_0: T_{\gamma (t)} M \rarrow T_{\Lambda^{\nablabar}_{A_0} (\gamma) (t)} \hM, \text{ such that for $t \in [0,1]$},\\
& (P^{\nablabar})_0^t (\gamma) A_0 := (P^{\hnabla^n})_0^t (\Lambda^{\nablabar}_{A_0} (\gamma) ) \circ A_0 \circ (P^{\nabla})_t^0 (\gamma)  =  A_0 \circ (P^{\nabla})_t^0 (\gamma).
\end{align*}
\end{definition}

We define the No-Spinning development lift of $(X,\hat{X}) \in T_{(x,\hx)} (M \times \R^n)$, the Rolling development lift and the Rolling development distribution of $X \in T_x M$ respectively as follows.

\begin{definition}
Let $\q \in Q$, $(X,\hat{X}) \in T_{(x,\hx)} (M \times \R^n)$ and $\gamma$ (resp. $\hgamma$) be an a.c. curve on $M$ (resp. on $\R^n$) starting at $x$ (resp. $\hx$) with initial velocity $X$ (resp. $\hX$). The No-Spinning development lift of $(X,\hat{X})$ is the unique vector $\lns(X, \hat{X}) |_{q}$ of $T_q Q$ at $\q$ given by
$$
\lns(X, \hX)|_q := \frac{d}{dt} \big|_0 (P^{\hnabla^n})_0^t (\hgamma) \circ A \circ (P^{\nabla})_{t}^0 (\gamma) = \frac{d}{dt} \big|_0 A \circ (P^{\nabla})_{t}^0 (\gamma).
$$
If, moreover, the initial velocity of $\hgamma$ is $AX$, then we define the Rolling lift $\lr $ at $\q \in Q$ to be the injective map from $T_x M$ onto $T_q Q$, such that for every $X \in T_{x}M$,
$$
\lr(X)|_q:= \lns (X, AX) |_{q} = \frac{d}{dt} \big|_0 (P^{\hnabla^n})_0^t (\hgamma) \circ A \circ (P^{\nabla})_{t}^0 (\gamma) = \frac{d}{dt} \big|_0 A \circ (P^{\nabla})_{t}^0 (\gamma).
$$
The Rolling distribution $\dr$ at $\q \in Q$ is an $n$-dimensional smooth distribution defined by
$$
\dr |_{q} := \lr (T_{x} M) |_{q}.
$$
\end{definition}

We say that an a.c. curve $t \mapsto q(t)=(\gamma(t), \hgamma(t); A(t))$ on $Q$, is a rolling curve if and only if it is tangent to $\dr$ for a.e. $t\in I$, where $I$ is a bounded interval of $\mathbb{R}$, i.e. if and only if $\dot{q} (t) = \lr (\dot{\gamma} (t)) |_{q(t)}$ for a.e. $t \in I$.
For the proof of next proposition, see \cite{ChitourKokkonen,KobayashiNomizu}.

\begin{proposition}\label{p3.1}
For any $q_0 := (x_0 , \hx_0 ; A_0) \in Q$ and any a.c. curve $\gamma : [0,1] \rarrow M$ starting at $x_0$, there exist  unique a.c. curves $ \hgamma (t) := \Lambda^{\nablabar}_{A_0} (\gamma) (t) $ and $A (t) := (P^{\nablabar})_0^t (\gamma) A_0$ such that $ A (t) \dot{\gamma} (t) = \dot{\hgamma} (t)$ and $\nablabar_{(\dot{\gamma} (t), \dot{\hgamma} (t))} A(t) = 0$, for all $t \in [0,T]$,
for a maximal $T=T(\gamma)$ such that $0<T\leq 1$.
We refer to $t \mapsto q_{\dr} (\gamma,q_0) : = (\gamma (t), \hgamma (t); A (t))$ as the rolling  curve along $\gamma$ with initial position $q_0$.

Moreover, if $(\hat{M},\hat{g})$ is a complete Riemannian manifold and $\hat{\nabla}$
is the corresponding Levi-Civita connection,
then one may take above $T=1$ for all $\gamma$'s.
\end{proposition}

Consider the smooth bundle $\pq: Q \rarrow M \times \R^n$ and $q \in Q$. We define $V |_{q} (\pq)$ to be the set of all $B \in T |_{q} Q$ such that the tangent application $(\pq)_{*} (B) = 0$. The collection of spaces $V |_{q} (\pq)$, $q \in Q$ defines a smooth submanifold $V (\pq)$ of $TQ$. We will write an element of $V |_{q} (\pq)$ at $\q \in Q$ as $\nu (B) \onq$ where $B \in T^\ast_x  M \otimes \R^n $ verifies $B \in A \; \mathfrak{so}(T_x M)$. Intrinsically, to know what it means to take the derivative with respect to $\nu (B) \onq$.
Then, for all smooth maps $f$ defined on (an open subset of) $Q$
with values in the manifold of $(m,k)$-tensors of $M$,
we define
\centerline{
$\nu (B) \mid_{q} (f):= \frac{d}{dt} |_0 f(q+tB),$
}
that we call the vertical derivative of $f$ at $q$ in the direction of $B$.

We next present the main computation tools obtained in Proposition 3.7, Lemma 3.18, Proposition 3.24, Proposition 3.26, Proposition 4.1, Proposition 4.6 \cite{Kokkonen2}.

\begin{proposition}
Let $\Oh \subset T^* M \otimes \mathbb{R}^n$ be an immersed submanifold, $\overline{Z} = (Z , \hZ)$, $\sbar = (S, \hS) \in C^{\infty} (\pi_{\Oh}, \pi_{T^* M \otimes \mathbb{R}^n})$ be such that for all $\q \in \Oh$, $\lns (\overline{Z} (q)) |_{q}$, $\lns (\sbar (q)) |_{q} \in T_{q} \Oh$ and $U$, $V \in C^{\infty} (\pi_{\Oh}, \pi_{T^* M \otimes \mathbb{R}^n})$, be such that for all $\q \in \Oh$, $\nu (U(q)) \onq$, $\nu(V(q)) \onq \in T_{q} \Oh$. Then, one has
\begin{eqnarray}\label{e3.23}
\leftline{
$
\lns(\overline{Z} (A)) |_{q} \sbar (\cdot) := \nablabar_{\overline{Z} (A)} (\sbar(A)),
$
}
\end{eqnarray}
\begin{eqnarray}\label{e3.12}
\leftline{
$
\begin{array} {rcl}
[\lns(\overline{Z} (\cdot)), \lns(\sbar(\cdot))] \onq& = & \lns (\lns(\overline{Z} (A)) \onq \sbar (\cdot) - \lns(\sbar (A)) \onq \overline{Z} (\cdot)) \onq \\
\\
&   - & \lns (T^{\nabla} (Z,S)) \onq + \nu (A R^{\nabla}(Z, S)) \onq,
\end{array}
$
}
\end{eqnarray}
\begin{eqnarray}\label{e3.33}
\leftline{
$
[\lr(Z), \lr(S)] \onq = \lr([Z,S]) \onq + \lns (A T^{\nabla} (Z,S)) \onq + \nu (A R^{\nabla}(Z(q), S(q))) \onq,
$
}
\end{eqnarray}
\begin{eqnarray}\label{e3.13}
\leftline{
$
[\lns(\overline{Z} (\cdot)), \nu(U(\cdot))] \onq  = - \lns (\nu(U (A)) \onq \overline{Z} (\cdot)) \onq + \nu (\lns( \overline{Z} (A)) \onq U (\cdot)) \onq,
$
}
\end{eqnarray}
\begin{eqnarray}\label{e3.14}
\leftline{
$
[\nu(U(\cdot)) , \nu(V(\cdot))] \onq = \nu (\nu(U(A)) \onq V - \nu(V(A)) \onq U)\onq.
$
}
\end{eqnarray}
Both sides of the equalities in \eqref{e3.23}, \eqref{e3.12}, \eqref{e3.33}, \eqref{e3.13} and \eqref{e3.14} are tangent to $\Oh$.
\end{proposition}

We use $\Aff (M)$ to denote the affine group of all invertible affine transformations from the affine manifold $M$ onto itself. In particular, the affine group of $\R^n$ is denoted by $\Aff(n)$. One can extend readily  Proposition 3.10 of \cite{ChitourKokkonen2} to get the following result.

\begin{proposition}\label{p3.9}
For any $f\in \Aff (M)$, $\hf \in \Aff (n)$ and any $\qz \in Q$, define the following smooth right and left actions of $\Aff (M)$ and $\Aff(n)$ on $Q$
$$
q_0 \cdot f := (f^{-1} (x_0), \hx_0; A_0 \circ f_\star |_{f^{-1} (x_0)} ), \quad \quad  \hf \cdot q_0 := (x_0, \hf( \hx_0) ; \hf_{\star} |_{\hx_0} \circ A_0).
$$
Then, for any a.c. curve $\gamma : [0,1] \rarrow M$ starting at $x_0$, one has for a.e. $t \in [0,1]$
$$
\hf \cdot q_{\dr} (\gamma, q_0) (t) \cdot f = q_{\dr} (f^{-1} \circ \gamma , \hf \cdot  q_0 \cdot f) (t).
$$
\end{proposition}

\begin{proof}
By the definition of an affine transformation $f$ on $M$, we have Eq. \eqref{e3.32} for any a.c. curve $\gamma :[0,1] \rarrow M$. This implies that, for a.e. $t \in [0,1]$
$$
f_\star |_{\gamma (t)} \circ (P^\nabla)_0^t (\gamma) = (P^\nabla)_0^t (f \circ \gamma) \circ f_\star  |_{\gamma (0)} .
$$
We have the same conclusion for affine transformations $\hf$ on $\R^n$. Then, since $\Aff (n)$ is a Lie group and by what precedes, one can repeat the steps of the proof of Proposition 3.10 in \cite{ChitourKokkonen2} with the group $\Aff (n)$ instead of isometry groups on $M$ and $\R^n$ to get the claim.
\end{proof}

Recall that if $G$ is a Lie group, then a smooth bundle $\pi: E \rarrow M$ is a principal $G$-bundle over $M$ if there exists a smooth and free action of $G$ on $E$ which preserves the fibers of $\pi$,  cf. \cite{Joyce}. Furthermore, we recall that the affine group $\Aff(n)$ is equal to $\mathbb{R}^n \rtimes GL(n)$ and its product group $\diamond$ is given by
$$
(v,L) \diamond (u , K) := (Lu+v, L \circ K).
$$

Using the previous proposition, one can extend immediately the simple but crucial Proposition 4.1 in \cite{ChitourKokkonen2} to derive the next result.

\begin{proposition}\label{p3.3}
The bundle $\pi_{Q,M}: Q \rarrow M$ is a principal $\Aff(n)$-bundle with the left action $\mu: \Aff(n) \times Q \rarrow Q$;
$$
\mu((\hy,C), (x,\hx;A)) = (x, C \hx + \hy ; C \circ A).
$$
The action $\mu$ preserves $\dr$, i.e. for any $q \in Q$ and $B \in \Aff(n)$, we have $(\mu_B)_{\ast} \dr |_q = \dr |_{\mu_B (q)}$ where $\mu_B :Q \rarrow Q$; $q \mapsto \mu(B,q)$. Moreover, for any $\q \in Q$, there exists a unique subgroup $\mathcal{H}^\nabla_q$ of $\Aff(n)$, called the affine holonomy group of $(M,\nabla)$ verifying
\[
\mu (\mathcal{H}^\nabla_q \times \{q\}) = \odr (q) \cap \pi_{Q,M}^{-1} (x).
\]
If $q' = (x,\hx';A') \in Q$ belongs to the same $\pi_{Q,M}$-fiber as $q$, then $\mathcal{H}^\nabla_q$ and $\mathcal{H}^\nabla_{q'}$ are conjugate in $\Aff(n)$ and all conjugacy classes of $\mathcal{H}^\nabla_q$ are of the form $\mathcal{H}^\nabla_{q'}$. This conjugacy class is denoted by $\mathcal{H}^\nabla$ and its projection in $GL(n)$ is equal to $H^{\nabla}$ the holonomy group of the affine connection $\nabla$. 
\end{proposition}

\begin{proof}
Let $\q \in Q$ and $B = (\hy,C) \in \Aff(n)$. Since $C \circ A$ is in $\GL(n)$, then $\mu(B,q) \in Q$. In order to prove that $\mu$ is transitive and proper, we can follow the same steps of the proof of Proposition 4.1 in \cite{ChitourKokkonen2} due to Proposition \ref{p3.9}.
\end{proof}

\subsubsection{Affine Holonomy Group of $\Delta$}

Consider now a smooth completely controllable distribution $\Delta$ on $(M,\nabla)$. We will determine the sub-distribution of $\dr$ by restriction to $\Delta$
instead of considering the whole tangent space of $M$.

\begin{definition}
The rolling distribution $\Dr$ on $\Delta$ is the smooth sub-distribution of $\dr$ defined on $(x,\hx;A) \in Q$ by
\begin{equation}\label{e3.15}
\Dr |_{(x, \hx; A)} = \lr (\Delta |_x) |_{(x,\hx;A)}.
\end{equation}
\end{definition}

Since $\Delta$ is completely controllable, we use Proposition \ref{p3.1} to obtain the next corollary.

\begin{corollary}
For any $q_0 = (x_0 , \hx_0 ; A_0) \in Q$ and any a.c. $\Delta$-admissible curve $\gamma : [0,1] \rarrow M$ starting at $x_0$, there exists a unique a.c. $\Dr$-admissible curve $q_{\Delta_R} (\gamma,q_0) : [0,T] \rarrow Q$ where $0<T\leq 1$.
\end{corollary}

Since we can easily restrict the proof of Proposition \ref{p3.3} (cf. \cite{ChitourKokkonen2}) on $\Dr$, we get the next proposition.

\begin{corollary}\label{c3.2}
The action $\mu$ mentioned in Proposition \ref{p3.3} preserves the distribution $\Dr$. Moreover, for every $q\in Q$, there exists a unique \emph{algebraic} subgroup $\mathcal{H}^{\;\nabla}_{\Dr |q}$ of $\mathcal{H}^\nabla_q$, called the affine holonomy group of $\Delta_R$, such that
$$
\mu (\mathcal{H}^{\;\nabla}_{\Dr |q} \times \{q\}) = \oDr (q) \cap \pi_{Q,M}^{-1} (x),
$$
where $x= \pi_{Q,M}(q)$ and $\oDr (q)$ is the $\Dr$-orbit at $q$.
\end{corollary}

As before, one gets the following: if $q' = (x,\hx';A') \in Q$ belongs to the same $\pi_{Q,M}$-fiber as $q$, then $\mathcal{H}^{\;\nabla}_{\Dr |q} $ and $\mathcal{H}^{\;\nabla}_{\Dr |{q'}} $ are conjugate in $\Aff(n)$ and all conjugacy classes of $\mathcal{H}^{\;\nabla}_{\Dr |q} $ are of the form $\mathcal{H}^{\;\nabla}_{\Dr |{q'}} $. This conjugacy class is denoted by $\mathcal{H}^{\;\nabla}_{\Dr } $ and its projection in $GL(n)$ is a subgroup of $H^{\nabla}$ which is equal to the 
$\Delta$-horizontal holonomy group associated with $\Delta$ and the affine connection $\nabla$.

\begin{definition}
We denote by $\Oh_{\Dr}^{loop} (q_0)$ the set of the end points of the rolling development curves with initial conditions any point $\qz$ and any a.c. $\Delta$-admissible loop at $x_0$, i.e., for $\qz \in Q$,
\begin{align*}
\Oh_{\Dr}^{loop} (q_0) = \{ q_{\Delta_R} (\gamma,q_0)(1) \mid \gamma : [0,1] \rarrow M, \text{ a.c. $\Delta$-admissible loop at $x_0$} \}.
\end{align*}
\end{definition}

If we fix a point $q_0$ of $Q = Q(M,\R^n)$ where the initial contact point on $M$ is equal to $x_0$ and that on $\R^n$ is the origin, then we may consider the rolling development of $M$ along a loop based at $x_0$. Then, one obtains a control problem whose state space is the fiber $\pi^{-1}_{Q,M}(x_0)$ and the reachable set is in the fiber $\pi^{-1}_{Q,M}(x_0)$ (for more details, cf. \cite{ChitourKokkonen2}). Then, $\Oh_{\Dr}^{loop} (q_0)$ is trivially in bijection with $\oDr (q_0) \cap \pi_{Q,M}^{-1} (x_0) $ and so $\mu (\mathcal{H}^{\;\nabla}_{\Dr |q_0} \times \{q_0\}) \simeq \Oh_{\Dr}^{loop} (q_0)$.

\begin{proposition}\label{pr:submersion}
For any $q_0=(x_0,\hat{x}_0;A_0)\in Q$ the restriction of $\pi_{Q,M}:Q\to M$
into the orbit $\oDr (q_0)$ is a submersion onto $M$.
\end{proposition}

\begin{proof}
Clearly it is enough to show that $(\pi_{Q,M})_*T_{q_0}\oDr (q_0)=T_{x_0} M$.
Also recall that by the assumption of complete controllability of $\Delta$ we have $M=\mc{O}_{\Delta}(x_0)$.

Write $E^{x,t}(u)$ and $\tilde{E}^{q,t}(u)$ for the end-point maps of $\Delta$ and $\Delta_R$
starting from $x\in M$ and $q\in Q$, respectively.
One easily sees that $E$ and $\tilde{E}$ are related by
\begin{align}\label{eq:endproj}
\pi_{Q,M}\circ \tilde{E}^{q,t}=E^{x,t},
\end{align}
for any $q=(x,\hat{x};A)\in Q$ and $t$ where defined.
We also denote by $k$ the rank of $\Delta$ (i.e. the rank of $\Delta_R$).

Let $\ol{u}\in L^2([0,1],\R^k)$ be any o-regular control of $E^{x_0,1}$
which belongs to the domain of definition of $\tilde{E}^{q_0,1}$.
The existence of such an $\ol{u}$ is guaranteed by an application
of Proposition~\ref{regular-0} given in the appendix
and Proposition \ref{p3.1}, as in this case $(\hat{M},\hat{g})=\R^n$ is complete.

Let then $X\in T_{x_0} M$ be arbitrary,
and notice that $T_{x_0} \mc{O}_{\Delta}(x_0)=T_{x_0} M$.
By o-regularity of $\ol{u}$ with respect to $E^{x_0,1}$,
there exists a $C^1$-map $u:I\to L^2([0,1],\R^k)$,
where $I$ is an open neighbourhood of $0$,
such that $u(0)=\ol{u}$ and $h(t,s):=E^{x_0,t}(u(s))$, $(t,s)\in [0,1]\times I$,
satisfy $\pa{s}h(1,s)|_{s=0}=X$.
Indeed, let $G:I\to \mc{O}_{\Delta}(x_0)$ be any smooth curve such that $\dot{G}(0)=X$.
The o-regularity of $\ol{u}$ means that $D_{u} E^{x_0,1}$, i.e. the differential of $E^{x_0,1}$ at $u$,
is surjective linear map from $L^2([0,1],\R^k)$ onto $T_{E^{x_0,1}(u)}\mc{O}_{\Delta}(x_0)$ when $u=\ol{u}$,
and hence for all $u$ close to $\bar{u}$ in $L^2([0,1],\R^k)$.
One next defines $P(u)$ as the Moore-Penrose inverse of $D_{u} E^{x_0,1}$ and one considers the Cauchy 
problem $\frac{du(s)}{ds}=P(u(s))\frac{dG(s)}{ds}$, $u(0)=\bar{u}$. Then \cite[Proposition 2]{chitour} asserts that the maximal solution $u(\cdot)$ of the Cauchy problem is well-defined on a non 
empty interval centered at zero, which concludes the argument of the claim (after shrinking $I$ if necessary).

Write $\tilde{h}(t,s)=\tilde{E}^{q_0,t}(u(s))$ for $(t,s)\in [0,1]\times I$.
For each $i=1,\dots,d$ and $s\in I$,
the maps $t\mapsto h(t,s)$ and $t\mapsto \tilde{h}(t,s)$ are absolutely continuous
and $\Delta$- and $\Delta_R$-admissible curves, respectively,
and $h(t,s)=\pi_{Q,M}(\tilde{h}(t,s))$ by \eqref{eq:endproj}.
In particular, $\pa{s}\tilde{h}(1,s)|_{s=0}$ is a vector in $T_{q_0}\oDr(q_0)$
and
\[
(\pi_{Q_,M})_*\big(\pa{s}\tilde{h}(1,s)|_{s=0}\big)=\pa{s}h(1,s)|_{s=0}=X,
\]
which shows that $X\in (\pi_{Q_,M})_*(T_{q_0}\oDr(q_0))$.
Because $X$ was arbitrary tangent vector of $M$ at $x_0$,
we conclude that
$T_{x_0} M\subset (\pi_{Q_,M})_*(T_{q_0}\oDr(q_0))$.

The opposite inclusion
$(\pi_{Q_,M})_*(T_{q_0}\oDr(q_0))\subset T_{x_0} M$ being trivially true,
this completes the proof.
\end{proof}

\begin{remark}
Here is an alternative proof
in the case that the distribution $\Delta$ satisfies LARC on a connected manifold $M$ i.e. $Lie_x(\Delta)=T_x M$ for all $x\in M$.

Given vector fields $Y_1,\dots,Y_r$
and a subset $J=\{i_1,\dots,i_l\}$ of $\{1,\dots,r\}$ we
write $Y_J$ for the iterated bracket $[Y_{i_1},[Y_{i_2},\dots[Y_{i_{l-1}},Y_{i_l}]\dots]$
of length $l$.
Given $X\in T_{x_0} M = T_{x_0} \mc{O}_{\Delta}(x_0)$,
there are, by the assumption, vector fields $Y_1,\dots,Y_r$ tangent to $\Delta$,
subsets $J_1,\dots,J_t$ of $\{1,\dots,r\}$
and numbers $a_1,\dots,a_t$ such that
$X=\sum_{s=1}^t a_s Y_{J_s}|_{x_0}$.
The lifts $\lr(Y_{i})$, $i=1,\dots,r$ are tangent to $\Delta_R$
and satisfy $(\pi_{Q,M})_*\lr(Y_{i})=Y_i$,
hence if we write $\lr(Y)_J$ for $[\lr(Y_{i_1}),[\lr(Y_{i_2}),\dots[\lr(Y_{i_{l-1}}),\lr(Y_{i_l})]\dots]$
when $J$ is as above, we have that $\lr(Y)_{J_s}$ is tangent to $\oDr(q_0)$ for every $s=1,\dots,t$
\[
(\pi_{Q,M})_*|_{q_0}\big(\sum_{s=1}^t a_s \lr(Y)_{J_s}\big)=\sum_{s=1}^t a_s Y_{J_s}|_{x_0}=X,
\]
i.e. $X\in (\pi_{Q,M})_*T_{q_0}\oDr(q_0)$. By arbitrariness of $X$ in $T_{x_0} M$
we have the claimed submersivity of $\pi_{Q,M}$.
\end{remark}

Classical results now apply to give the following.

\begin{corollary}\label{cor:fiber_submanifold}
In particular, for any $x\in M$ the fiber $\pi_{Q,M}^{-1}(x)\cap \oDr (q_0)$
of $\oDr (q_0)$ over $x$ is either empty or a (closed) embedded submanifold of $\oDr (q_0)$
of dimension $\delta=\dim \oDr (q_0)-\dim M$.
\end{corollary}

We arrive at the main result of this subsection.

\begin{proposition}\label{p3.10}
Assume that $\Delta$ is a constant rank completely controllable distribution on $(M,\nabla)$ where $M$ is a connected smooth manifold and $\nabla $ an affine connection. Then, the $\Delta$-horizontal holonomy group $H_{\Delta}^{\;\nabla}$  and the affine holonomy group $\mathcal{H}^{\;\nabla}_{\Dr}$ of $\Delta_R$ as defined previously are Lie subgroups of $\Aff(n)$. 
\end{proposition}

\begin{proof}
It is enough to prove the claim for $\mathcal{H}^{\;\nabla}_{\Dr}$.
We first argue that $\mathcal{H}^{\;\nabla}_{\Dr |q_0}$ is an algebraic subgroup of $\Aff(n)$.
To this end, to any $p\in \pi_{Q,M}^{-1} (x_0)$ (i.e. $p$ is an arbitrary element of
the fiber of $Q$ over $x_0$)
we match a unique $(y_p,C_p)\in \Aff(n)$ such that $\mu((y_p,C_p),q_0)=p$.
Recall that $\oDr (q_0) \cap\pi_{Q,M}^{-1} (x_0)$ is identified with $\mathcal{H}^{\;\nabla}_{\Dr |q_0}
$ through this correspondence.

Then given $p_1, p_2\in \oDr (q_0) \cap\pi_{Q,M}^{-1} (x_0)$,
there are $\Delta$-admissible (piecewise smooth) loops $\gamma_1,\gamma_2\in\Omega_M(x_0)$ in $M$
based at $x_0$
such that $p_i=q_{\Delta_R}(\gamma_i,q_0)(1)$ for $i=1,2$.
Letting $p=q_{\Delta_R}(\gamma_1\cdot\gamma_2,q_0)(1)$ we have
\[
\mu((y_p,C_p),q_0)
=&p=q_{\Delta_R}(\gamma_1\cdot\gamma_2,q_0)(1)
=q_{\Delta_R}\big(\gamma_1,q_{\Delta_R}(\gamma_2,q_0)\big)(1)
=q_{\Delta_R}\big(\gamma_1,p_2\big)(1) \\
=&q_{\Delta_R}\big(\gamma_1,\mu((y_{p_2},C_{p_2}),q_0)\big)(1)
=\mu((y_{p_2},C_{p_2}),q_{\Delta_R}\big(\gamma_1,q_0)(1)\big) \\
=&\mu((y_{p_2},C_{p_2}),p_1)
=\mu\big((y_{p_2},C_{p_2}),\mu((y_{p_1},C_{p_1}),q_0)\big) \\
=&\mu\big((y_{p_2},C_{p_2})(y_{p_1},C_{p_1}),q_0\big),
\]
i.e. $(y_p,C_p)=(y_{p_2},C_{p_2})(y_{p_1},C_{p_1})$,
because the action $\mu$ is free.
Since $\gamma_1\cdot\gamma_2$ is $\Delta$-admissible loop, we have
$p=q_{\Delta_R}(\gamma_1\cdot\gamma_2,q_0)(1)\in \oDr (q_0) \cap\pi_{Q,M}^{-1} (x_0)$
i.e. $(y_p,C_p)\in \mathcal{H}^{\;\nabla}_{\Dr |q_0}$,
and therefore $\mathcal{H}^{\;\nabla}_{\Dr |q_0}$ is indeed an algebraic subgroup of $\Aff(n)$
as claimed.

In other words we have shown that if $m:\Aff(n)\times\Aff(n)\to\Aff(n)$
is the smooth group multiplication operation on $\Aff(n)$,
then
\[
m(\mathcal{H}^{\;\nabla}_{\Dr |q_0}\times \mathcal{H}^{\;\nabla}_{\Dr |q_0})\subset \mathcal{H}^{\;\nabla}_{\Dr |q_0}.
\]
By the orbit theorem \ref{th:orbit} as given in the appendix (see also \cite{harms12}),
we know that any smooth map $f:Z\to Q$ for any smooth manifold $Z$
such that $f(Z)\subset \oDr (q_0)$
is smooth as a map $f:Z\to \oDr (q_0)$.
In other words, $\oDr (q_0)$ is an \emph{initial submanifold} of $M$ (cf. \cite{harms12}).

By Corollary \ref{cor:fiber_submanifold}
$\oDr (q_0)\cap\pi_{Q,M}^{-1} (x_0)$ is a smooth embedded submanifold of $\oDr (q_0)$,
hence an initial submanifold of $Q$.
Since $\oDr (q_0)\cap\pi_{Q,M}^{-1} (x_0)\subset \pi_{Q,M}^{-1} (x_0)$
and $\pi_{Q,M}^{-1} (x_0)$ is diffeomorphic to $\Aff(n)$ using the action $\mu$,
we have that $\mathcal{H}^{\;\nabla}_{\Dr |q_0}$ is a smooth immersed submanifold of $\Aff(n)$
as well.
Now the group multiplication $m$ restricted to $\mathcal{H}^{\;\nabla}_{\Dr |q_0}$
which we write as $m'$ is a smooth map $m':\mathcal{H}^{\;\nabla}_{\Dr |q_0}\times \mathcal{H}^{\;\nabla}_{\Dr |q_0}\to\Aff(n)$ whose image is a subset of $\mathcal{H}^{\;\nabla}_{\Dr |q_0}$.
Pulling this map back by the action $\mu$ on $Q$
we obtain a smooth map $M:(\oDr (q_0)\cap\pi_{Q,M}^{-1} (x_0))\times (\oDr (q_0)\cap\pi_{Q,M}^{-1} (x_0))\to Q$
whose image is contained in $\oDr (q_0)\cap\pi_{Q,M}^{-1} (x_0)$.
As mentioned above, $\oDr (q_0)\cap\pi_{Q,M}^{-1} (x_0)$ is an initial submanifold
of $Q$, hence $M$ is smooth as a map into $\oDr (q_0)\cap\pi_{Q,M}^{-1} (x_0)$.
This then is reflected, by applying the action $\mu$ once more,
in the fact that $m'$ is smooth as a map into $\mathcal{H}^{\;\nabla}_{\Dr |q_0}$.
Thus the latter space is a Lie-subgroup of $\Aff(n)$.
\end{proof}

\begin{remark}
The situation described in Remark~\ref{rieman0} with the rolling formalism can be treated as the 
rolling system without spinning nor slipping of two oriented connected Riemannian manifolds $(M, g)$ 
and $(\R^n, s_n)$, where $s_n$ is the Euclidean metric on $\R^n$. Thus, the state space $Q(M,\R^n)$ 
is a principal $\SE(n)$-bundle (cf \cite{ChitourKokkonen}, \cite{ChitourKokkonen1} and 
\cite{ChitourKokkonen2} for more details).
\end{remark}

\subsection{Integrability of $\Delta_R$}

A natural question arises in the framework of horizontal holonomy,
namely under which conditions the horizontal holonomy
group $\mc{H}_{\Delta_R}^\nabla$ is trivial.
More generally, we pose this question on the level of Lie algebra,
which translates on the group level to asking when $\mc{H}_{\Delta_R}^\nabla$ is discrete
in its underlying Lie group topology.
As $\mc{H}_{\Delta_R}^\nabla$ is identified with a fiber $\oDr (q_0) \cap\pi_{Q,M}^{-1} (x_0)$ of the orbit $\oDr (q_0)$,
we see that answering question comes down to
studying when the distribution $\Delta_R$ itself is involutive.

In the case where $\Delta=TM$, it is known that the answer is that
$\mc{H}_{\Delta_R}^\nabla$ is discrete if and only if $(M,\nabla)$ has vanishing curvature and torsion.
This justifies the following definition.

\begin{definition}
We say that the triple $(M,\nabla)$ is \emph{$\Delta$-horizontally flat} 
is $\Delta_R$ is involutive.
\end{definition}

By \eqref{e3.12}, we see that for any vector field $X,Y$ tangent to $\Delta$ we have
for any $q=(x,\hat{x};A)\in Q$,
\[
[\lr(X),\lr(Y)]|_q=\lr([X,Y])|_q-\lns(T^{\nabla}(X,Y))|_q+\nu(AR^{\nabla}(X,Y))|_q
\]
where, as before, $T^{\nabla}$ and $R^{\nabla}$ are the torsion and the curvature
of $\nabla$, respectively.

This formula immediately implies a simple characterization of the involutivity of $\Delta_R$.

\begin{proposition}
The manifold with connection $(M,\nabla)$ is $\Delta$-horizontally flat
if and only if $\Delta$ is involutive and
for all $x\in M$ and $X,Y\in \Delta|_x$,
\[
T^{\nabla}(X,Y)=0,\quad
R^{\nabla}(X,Y)=0
\]
\end{proposition}

For the rest of this subsection, we assume that $M$ is a Riemannian
manifold with metric $g$ and that $\nabla$ is the associated Levi-Civita connection.

Let $\Delta^\perp$ be the $g$-orthogonal complement of $\Delta$,
and let $P:TM\to \Delta$ be $P^\perp:TM\to \Delta^\perp$ be the orthogonal projections onto
$\Delta$ and $\Delta^\perp$, respectively.
Define the fundamental $\fII$ form of $\Delta$ by
\[
\fII(X,Y)=P^\perp(\nabla_X Y),
\quad \forall X,Y\in \Delta|_x,\ x\in M.
\]
When $\xi\in \Delta^{\perp}|_x$ is given,
one defines the shape operator $S_{\xi}:\Delta|_x\to\Delta|_x$ of $\Delta$ with respect to $\xi$ to be
given by
\[
g(S_{\xi}(X),Y)=-g(\xi,\fII(X,Y)),\quad \forall X,Y\in \Delta|_x.
\]

In the case where $\Delta$ is involutive $\fII$ is symmetric,
and we define the \emph{induced $\Delta$-connection} $D$
and \emph{induced $\Delta^\perp$-connection} $D^\perp$ by
\[
D_X Y&=P(\nabla_X Y), \\
D^\perp_X \xi&=P^\perp(\nabla_X \xi),
\]
for $X\in \Delta|_x$, for any vector field $Y$ tangent to $\Delta$
and for any vector field $\xi$ tangent to $\Delta^\perp$.
Furthermore, if one defines for $X,Y,Z\in \Delta|_x$, $\xi\in \Delta^\perp|_x$, where $x\in M$,
\[
R^D(X,Y)Z&=D_X D_Y Z-D_Y D_X Z-D_{[X,Y]} Z, \\
R^{\perp}(X,Y)\xi&=D^{\perp}_X D^{\perp}_Y \xi-D^{\perp}_Y D^{\perp}_X \xi-D^{\perp}_{[X,Y]} \xi,
\]
then the following result holds.

\begin{corollary}
The Riemannian manifold $(M,g)$ is $\Delta$-horizontally flat
if and only if $\Delta$ is involutive and
for all $X,Y,U,V\in \Delta|_y$, $\xi,\eta\in\Delta^\perp|_y$ and $y\in M$,
\begin{align}
g(R^D(X,Y)U,V)&=g\big(\fII(X,U),\fII(Y,V)\big)-g\big(\fII(X,V),\fII(Y,U)\big), \label{eq:gauss} \\
g((\nabla_X \fII)(Y,U),\xi)&=g((\nabla_Y \fII)(X,U),\xi), \label{eq:codazzi} \\
g(R^\perp(X,Y)\xi,\eta)&=g(S_{\xi}(X),S_{\eta}(Y))-g(S_{\eta}(X),S_{\xi}(Y)). \label{eq:ricci}
\end{align}
\end{corollary}

\begin{proof}
Indeed, if $L$ is a leaf of $\Delta$, $h$ is the metric on $L$ induced by $g$,
then $D$ is exactly the Levi-Civita connection of $h$,
and $\fII$ restricted to $L$ is the second fundamental form of $(L,h)$ in $(M,g)$.
The result follows from these observations combined (\cite[Theorem 1.72]{besse87}) with the Gauss, Codazzi-Mainardi and Ricci equations,
which are \eqref{eq:gauss}, \eqref{eq:codazzi} and \eqref{eq:ricci}, respectively.
\end{proof}

\section{Case Study: Holonomy of Free Step-two Homogeneous Carnot Group}
The goal of this section is to provide an example of a triple $(M,\nabla,\Delta)$ such that $\Delta$ verifies the LARC (and thus is completely controllable) and $H^{\;\nabla}_{\Delta}$ is a Lie group strictly included in $H^{\nabla}$. After giving the required definitions to treat the example, we first compute $H^{\nabla}$
and then $H^{\;\nabla}_{\Delta}$ using the rolling formalism.

\subsection{Definitions}

The affine manifold $(M,\nabla)$ we consider is the free step-two homogeneous Carnot group $\G$ endowed with a Riemannian metric and its Levi-Civita connection. To describe it, we will use the definitions of Jacobian basis, homogeneous group and Carnot group of Chapters 1 and 2 of \cite{BonfiglioliLanconelliUguzzoni}.

For  $m$ positive integer greater than or equal to $2$, set  $m+n$ where $n := m(m-1)/2$ and
$\I:= \{(h,k) \mid 1 \leq k< h \leq m\}$ of cardinal $n$. Let $S^{(h,k)}$ be the $m\times m$ real skew-symmetric matrix whose entries are $-1$ in the position $(h,k),$ $+1$ in the position $(k,h)$ and $0$ elsewhere.
On $\R^{m+n}$ where an arbitrary point is written $(v,\gamma)$ with $v \in \R^m$, and $ \gamma \in \R^n$, define the group law $\star$ by setting
\begin{equation}\label{e3.16}
(v,\gamma) \star (v',\gamma') =  \left (\begin{array}{c}
v_i +v'_i, \quad i=1,\dots,m\\
\gamma_{h,k} + \gamma'_{h,k}+ \frac{1}{2}(v_h v'_k - v_k v'_h), \quad (h,k)\in \I
\end{array}\right).
\end{equation}
Then it is easy to verify that $\G:=( \R^{m+n},\star)$ is a Lie group, more precisely a free step-two homogeneous Carnot group of $m$ generators.
Indeed, a trivial computation shows that the dilation $\delta_\lambda$ given by
\begin{equation}\label{e3.2}
\delta_\lambda: \R^{m+n} \rightarrow \R^{m+n}; \quad \delta_\lambda(v,\gamma)=(\lambda v, {\lambda}^2 \gamma),
\end{equation}
is an automorphism of $\G$ for every $\lambda > 0$.
On the other hand, the (Jacobian) basis of the Lie algebra $\mathfrak{g}$ of $\G$ is given by
$X_h$, $\Gamma_{h,k}$ where

\begin{align*}
 X_h &= \frac{\partial}{\partial v_h} + \frac{1}{2} \sum_{1 \leq j<i \leq m} \left (\begin{array}{c}
\sum_{l=1}^m {S^{(i,j)}_{h,l} v_l}
\end{array}\right)(\frac{\partial}{\partial \gamma_{i,j}}), \\
\\
&=  \left\{\begin{array}{lll}
\frac{\partial}{\partial v_1} + \frac{1}{2} \sum_{1 < i \leq m} v_i \frac{\partial}{\partial \gamma_{i,1}}  \quad & \hbox{if} \; h=1,\\
\\
\frac{\partial}{\partial v_h} + \frac{1}{2} \sum_{h < i \leq m} v_i \frac{\partial}{\partial \gamma_{i,h}} - \frac{1}{2} \sum_{1 \leq j < h} v_j
\frac{\partial}{\partial \gamma_{h,j}}  \quad & \hbox{if} \; 1<h<m,\\
\\
\frac{\partial}{\partial v_m} - \frac{1}{2} \sum_{1\leq j < m} v_j \frac{\partial}{\partial \gamma_{m,j}}  \quad & \hbox{if} \; h=m,
\end{array} \right.\\
\\
\Gamma_{h,k} &=\frac{\partial}{\partial \gamma_{h,k}}, \quad \quad (h,k) \in \I.
\end{align*}
while the Lie brackets on $\G = (\R^N, \star)$ are given by
\begin{align*}
[X_h,X_k] &= \sum_{1 \leq j<i \leq m} {S^{(i,j)}_{h,k} \frac{\partial}{\partial\gamma_{i,j}}} = \frac{\partial}{\partial\gamma_{h,k}}= \Gamma_{h,k} \\
[X_h,\Gamma_{i,j}]&=0,\ 
[\Gamma_{h,k},\Gamma_{i,j}]=0.
\end{align*}
Then,
$$ \hbox{rank}(Lie\{X_1,\dots,X_m\})=\hbox{dim}(\hbox{span}\big\{\frac{\partial}{\partial v_1},\dots,\frac{\partial}{\partial v_m},(\Gamma_{h,k})_{(h,k) \in \I}\big\}) = N = \dim \mathfrak{g}.
$$
Therefore, we can conclude that $\G$ is a homogeneous Carnot group of step $2$ and $m$ generators $X_1,\dots,X_m$. The Lie algebra $\mathfrak{g}$ is equal to $V_1 \oplus V_2$, where $V_1 = \text{ span} \{X_1 , \dots, X_m\}$ and $V_2 = \text{ span} \{\Gamma_{h,k}, \; (h,k) \in \I\}$.

Moreover, $(\G,g)$ is an analytic manifold where the metric $g$, with respect to the previous basis, is given by
\begin{align}\label{e3.24}
\begin{array}{lll}
&g(X_i,X_j) = \delta_{i,j}, \quad & \text{if} \; \; i,j \in \{1,\dots,m\},\\
\\
&g(X_i, \Gamma_{h,k})=0 , \quad & \text{if} \;\;  i \in \{1,\dots,m\} \; \; \text{and} \;\; (h,k)\in \I,\\
\\
&g(\Gamma_{h,k},\Gamma_{i,j})=\delta_{h,i} \delta_{k,j}, \quad & \text{if} \; \; (i,j), (h,k) \in \I.
\end{array}
\end{align}

In the sequel of this article, we find useful to introduce the following notation of vector fields instead of $\Gamma_{h,k}$, for $ h,k \in \{1,\dots,m\}$, in order to facilitate computations by avoiding the confusion between the two cases $k <h$ and $h <k$.

\begin{definition}
For every $ h,k \in \{1,\dots,m\}$, we define,
\begin{equation}\label{e3.17}
\Omega_{h,k} = \left\{\begin{array}{cl}
\Gamma_{h,k} \quad & \text{if}\; \; h>k,\\
- \Gamma_{k,h} \quad & \text{if} \; \; h<k,\\
0 \quad & \text{if} \; \; h=k.
\end{array} \right.
\end{equation}
\end{definition}

By the above definition, the Lie bracket $[X_h,X_k]$ is equal to $\Omega_{h,k},$ for any $ h,k \in \{1,...,m\}$. Furthermore, let $\nabla^g$ be the Levi-Civita connection associated to the Riemannian metric in \eqref{e3.24}.

\begin{lemma}\label{l3.1}
For $h,k,l,s,t \in \{1,\dots,m\}$, we have the following covariant derivatives on $(\G,g)$,
\begin{align*}
&\nabla^g_{X_h} X_k = \frac{1}{2} \Omega_{h,k},& &\nabla^g_{\Omega_{h,k}}{\Omega_{s,t}} = 0, \\
&\nabla^g_{X_l}{\Omega_{h,k}} = \frac{1}{2} (\delta_{kl}X_h - \delta_{hl}X_k),& &\nabla^g_{\Omega_{h,k}}{X_l} = \frac{1}{2} (\delta_{kl}X_h - \delta_{hl}X_k).
\end{align*}
\end{lemma}

\begin{proof}
Let us denote by $\nabla^g_X Y$ the covariant differential of a vector field $Y$ in the direction of another vector field $X$ on $\G$. It is equal to
\begin{align}\label{e3.25}
\nabla^g_XY = \sum_{h=1}^{m} {\alpha_h(X,Y) X_h} + \sum_{1 \leq k<h \leq m}{\beta_{(h,k)}(X,Y) \Omega_{h,k}}.
\end{align}
On the other hand, by Koszul's formula (cf. \cite{Sakai}), we have
\begin{align}\label{e3.26}
2 g (\nabla^g_X Y , Z) = g ([X,Y] ,Z) - g ([X,Z],Y) - g ([Y,Z],X).
\end{align}
Combining \eqref{e3.25} and \eqref{e3.26}, we easily find the coefficients $\alpha_h (X,Y)$ and $\beta_{(h,k)} (X,Y)$ and hence we obtain the claim.
\end{proof}

\subsection{Riemannian Holonomy Group of $(\G,g)$}

The main of this subsection is to prove the following theorem.

\begin{theorem}\label{t3.2}
Let $(\G,\nabla^g)$ be a free step-two homogeneous Carnot group of dimension $N$ endowed with the Levi-Civita connection $\nabla^g$ given in Lemma \ref{l3.1}. Then, $(\G,\nabla^g)$ has full holonomy group $H^{\nabla^g} = \SO(m+n)$.
\end{theorem}

To this end, we compute the Riemannian tensor curvature $R$ and as well as part of its covariant derivation of $(\G,\nabla^g)$.

\begin{lemma}\label{l3.2}
For any $h,k,l,i,j \in \{1,\dots,m\}$, the Riemannian curvature tensor $R$ of $(\G, \nabla^g))$ is given by the following skew-symmetric matrices,
\begin{align}
&R(X_h , X_k) = \frac{3}{4} (X_h \wedge  X_k) + \frac{1}{4} \sum_{j=1}^{m} \Omega_{h,j} \wedge \Omega_{k,j}, \label{e3.5}
\\
&R(X_l,\Omega_{h,k}) = \frac{1}{4} ( X_h \wedge \Omega_{k,l} + X_k \wedge \Omega_{l,h} ), \label{e3.6}
\\
&R (\Omega_{i,j} , \Omega_{h,k}) = \frac{1}{4} (\delta_{ik} X_h \wedge X_j + \delta_{jk} X_i \wedge X_h + \delta_{ih} X_j \wedge X_k + \delta_{jh} X_k \wedge X_i). \label{e3.7}
\end{align}
\end{lemma}

\begin{proof}
From Lemma \ref{l3.1} and the intrinsic definition of $R$,
$$
R(X,Y)Z = \nabla^g_X \nabla^g_Y Z - \nabla^g_Y \nabla^g_X Z - \nabla^g_{[X,Y]}Z, \quad \quad \forall X,\; Y,\; Z \in T_x \G,
$$
we get, for any $h,k,l,i,j \in \{1, \dots, m\}$,
\begin{eqnarray*}
\begin{array}{l}
R(X_h, X_k)X_l = \frac{3}{4} (\delta_{hl}X_k - \delta_{kl}X_h  ) , \\
\\
R(X_h, X_k)\Omega_{i,j}=\frac{1}{4} ( \delta_{ih} \Omega_{k,j} + \delta_{jh} \Omega_{i,k} - \delta_{ik} \Omega_{h,j} - \delta_{jk} \Omega_{i,h}).
\end{array}
\end{eqnarray*}
Similarly, for any $h,k,l, i,j,t \in \{1,\dots,m\}$, $R(X_l, \Omega_{h,k})$ is given by
\begin{eqnarray*}
\begin{array}{l}
R(X_l,\Omega_{h,k})X_t =\frac{1}{4} ( \delta_{th} \Omega_{k,l} - \delta_{tk} \Omega_{h,l} ),\\
\\
R(X_l,\Omega_{h,k})\Omega_{i,j}=\frac{1}{4} ((\delta_{jk}\delta_{il}-\delta_{jl}\delta_{ki})X_h
+(\delta_{jl}\delta_{hi}-\delta_{il}\delta_{jh})X_k).
\end{array}
\end{eqnarray*}
Finally, for any $i,j,h,k,l \in \{1,\dots,m\}$, $R(\Omega_{i,j},\Omega_{h,k})$ is given by
\begin{align*}
R(\Omega_{i,j},\Omega_{h,k})X_l = \frac{1}{4} & \big( (\delta_{lk}\delta_{jh} - \delta_{hl}\delta_{jk} ) X_i + ( \delta_{hl}\delta_{ik} - \delta_{lk}\delta_{ih} ) X_j\\
&  + (\delta_{il}\delta_{jk} - \delta_{jl}\delta_{ik} ) X_h + ( \delta_{lj}\delta_{ih} - \delta_{hj}\delta_{il} ) X_k \big),\\
\\
R(\Omega_{i,j},\Omega_{h,k})\Omega_{s,t}=0, &  \quad \quad \forall s,t \in \{1,\dots,m\}.
\end{align*}
Collecting the above equalities, we get Eq. \eqref{e3.5}, Eq. \eqref{e3.6} and Eq. \eqref{e3.7}.
\end{proof}

Using the definition of the covariant derivative of tensors, which is,
$$
(\nabla^g_ZR(X,Y))(W) = \nabla^g_Z(R(X,Y)W) - R(X,Y)\nabla^g_Z W, \quad \quad \forall X,\; Y,\; Z,\; W \in T_x \G,
$$
we deduce the following lemma.

\begin{lemma}\label{l3.3}
The covariant derivatives of $R$ in the direction of a vector fields $X_t$ on $\G$, for $t \in \{1,\dots,m\}$, are
\begin{align*}
&\nabla^g_{X_t} R(X_h,X_k)= -R(X_t, \Omega_{h,k}) + \frac{1}{8}  \sum_{j=1}^{m} \big( \delta_{kt} X_j \wedge \Omega_{h,j} - \delta_{ht} X_j \wedge \Omega_{k,j} \big),\\
&\nabla^g_{X_t} R(X_l,\Omega_{h,k})= \frac{1}{8} \big( \Omega_{t,h} \wedge \Omega_{k,l}  + \Omega_{t,k} \wedge \Omega_{l,h} + 2 \delta_{lt} X_h \wedge X_k + \delta_{ht} X_k \wedge X_l - \delta_{kt} X_h \wedge X_l  \big),\\
&\nabla^g_{X_t} R(\Omega_{i,j},\Omega_{h,k})= \frac{1}{8} \big(  \delta_{ik} R(X_j, \Omega_{h,t}) + \delta_{jk} R(X_h, \Omega_{i,t}) + \delta_{ih} R(X_k, \Omega_{j,t}) +\delta_{jh} R(X_i, \Omega_{k,t})\big),
\end{align*}
where $h,k,l,i,j$ are any integers in $\{1,\dots,m\}$.

Similarly, the covariant derivatives of $R$ in the direction of a vector fields $\Omega_{s,t}$ on $\G$, for every $s,t \in \{1,\dots,m\}$, are
\begin{align*}
\nabla^g_{\Omega_{s,t}} R(X_h,X_k)&= \frac{3}{8} \big( \delta_{th} X_s \wedge X_k - \delta_{sh} X_t \wedge X_k + \delta_{tk} X_h \wedge X_s - \delta_{sk} X_h \wedge X_t \big),\\
\nabla^g_{\Omega_{s,t}} R(X_l,\Omega_{h,k})&= \frac{1}{8} \big( \delta_{th} X_s \wedge \Omega_{k,l} - \delta_{sh} X_t \wedge \Omega_{k,l} + \delta_{tk} X_s \wedge \Omega_{l,h} - \delta_{sk} X_t \wedge \Omega_{l,h}\big),\\
\nabla^g_{\Omega_{s,t}} R(\Omega_{i,j},\Omega_{h,k})
&= \frac{1}{8} \big( (\delta_{ih}\delta_{tk}-\delta_{ik}\delta_{th}) X_j \wedge X_s + (\delta_{ik}\delta_{jt}-\delta_{jk}\delta_{ti}) X_h \wedge X_s \\
& \quad \quad + (\delta_{jh}\delta_{ti}-\delta_{ih}\delta_{tj}) X_k \wedge X_s + (\delta_{jk}\delta_{th}-\delta_{jh}\delta_{tk}) X_i \wedge X_s \\
& \quad \quad - (\delta_{ik}\delta_{js}-\delta_{jk}\delta_{is}) X_h \wedge X_t - (\delta_{jk}\delta_{sh}-\delta_{jh}\delta_{sk}) X_i \wedge X_t \\
& \quad \quad -(\delta_{jh}\delta_{si}-\delta_{ih}\delta_{sj}) X_k \wedge X_t -(\delta_{ih}\delta_{sk}-\delta_{ik}\delta_{sh}) X_j \wedge X_t \big),
\end{align*}
where $h,k,l,i,j$ are any integers in $\{1,\dots,m\}$.
\end{lemma}
We next deduce from the two previous lemma the main computational result of the section.
\begin{proposition}\label{p3.7}
Fix some $q_0 \in Q(\G,\R^N)$ and let $\q \in \odr (q_0)$, then $\SO(T_x M) \subset \odr (q_0)$.
\end{proposition}

\begin{proof}
Fix some $\qz \in Q$, for any $h,k,i,j \in \{1,\dots,m\}$ such that $i \neq j$ and $k\neq h$, the first order Lie brackets on $\odr(q_0)$ are
\begin{align*}
[\lr(X_h),\lr(X_k)] \onq &= \lr(\Omega_{h,k}) \onq + \nu ( A R (X_h, X_k)) \onq\\
&=  \lr(\Omega_{h,k}) \onq + \frac{3}{4} \nu ( A (X_h \wedge X_k)) \onq + \frac{1}{4} \nu \big( A (\sum_{j=1}^{m} \Omega_{h,j}  \wedge \Omega_{k,j} ) \big) \onq,\\
[\lr(\Omega_{i,j}),\lr(\Omega_{h,k})] \onq & = \frac{1}{4} \nu (A (\delta_{ik} X_h \wedge X_j + \delta_{jk} X_i \wedge X_h + \delta_{ih} X_j \wedge X_k + \delta_{jh} X_k \wedge X_i) ) \onq,\\
[\lr(X_i),\lr(\Omega_{h,k})] \onq &= \frac{1}{4} \nu ( A (X_h \wedge \Omega_{k,i}  + X_k \wedge \Omega_{i,h} )) \onq.
\end{align*}
By taking $i=k$ in the bracket $[\lr(\Omega_{i,j}),\lr(\Omega_{h,k})] \onq$, we get that, for any $h,j \in \{1,\dots,m\}$, $\nu ( A (X_h \wedge X_j )) \onq$ is tangent to $\odr(q_0)$. In addition, from the first and the last brackets of the above Lie brackets, we obtain that $\nu \big( A (\sum_{j=1}^{m} \Omega_{h,j}  \wedge \Omega_{k,j} ) \big) \onq$ and $\nu ( A (X_h \wedge \Omega_{k,i}  + X_k \wedge \Omega_{i,h} )) \onq$ are tangent to $\odr(q_0)$, for any $h,k,i,j \in \{1,\dots,m\}$. Thus, we can compute the next bracket in $T_q \odr (q_0)$, for $q \in \odr(q_0)$,

\begin{eqnarray*}
\leftline{
$
\begin{array}{rcl}
[\lr(X_i),\nu ((\cdot) ( X_h \wedge X_k))] \onq & = & \delta_{ki} \lns (A X_h) \onq - \delta_{hi} \lns (A X_k) \onq \\
\\
& - & \frac{1}{2} \nu ( A (X_h \wedge \Omega_{k,i}  + X_k \wedge \Omega_{i,h} )) \onq.
\end{array}
$
}
\end{eqnarray*}
Using $[\lr(X_i),\lr(\Omega_{h,k})] \onq$ and then putting $i=h$ in the last Lie bracket, we obtain that $\lns (A X_k) \onq$ is tangent to $\odr(q_0)$, for all $k \in \{1,\dots,m\}$. In addition, we have
\begin{eqnarray*}
\leftline{
$
\begin{array}{cl}
& [\lr(\Omega_{t,s}), \nu ( (\cdot) (X_h \wedge \Omega_{k,i}  + X_k \wedge \Omega_{i,h} ))] \onq\\
\\
= & (\delta_{tk} \delta_{ls} - \delta_{tl}\delta_{sk}) \lns (A X_h) \onq + (\delta_{ht}\delta_{ls} - \delta_{tl}\delta_{hs}) \lns (A X_k) \onq \\
\\
+ & \frac{1}{2} \big( \delta_{sh} \nu (A ( X_t \wedge \Omega_{k,l})) \onq + \delta_{sk} \nu (A ( X_t \wedge \Omega_{l,h})) \onq \\
\\
& \quad - \delta_{th} \nu (A ( X_s \wedge \Omega_{k,l})) \onq + \delta_{tk} \nu (A ( X_s \wedge \Omega_{l,h} )) \onq \big).
\end{array}
$
}
\end{eqnarray*}

Since $\lns (A X_k) \onq$ is tangent to $\odr(q_0)$, for all $k \in \{1,\dots,m\}$, then $\nu (A (X_t \wedge \Omega_{h,k})) \onq$ is also tangent for any distinct integers $h,k,t \in \{1,\dots,m\}$. The last Lie bracket to compute is
\begin{align*}
[\lns(X_t),\nu  \big(A ( X_l \wedge \Omega_{h,k} ) \big)] \onq = \frac{1}{2} \big( & \delta_{tk}\nu (A ( X_l \wedge X_h))\onq -  \delta_{th}\nu (A ( X_l \wedge X_k ))\onq\\
& + \nu (A ( \Omega_{t,l} \wedge \Omega_{h,k}))\onq  \big).
\end{align*}
Therefore, for every $h, k, t, l \in \{ 1,\dots,m\},$ $\nu (A ( \Omega_{t,l} \wedge \Omega_{h,k}))\onq $ is tangent to $\odr(q_0).$ Hence, for all $q \in \odr(q_0)$ the following vector fields
\begin{eqnarray*}
\nu ( A (X_h \wedge X_k)) \onq,\;\; \nu ( A ( X_t  \wedge \Omega_{h,k} )) \onq,\;\; \nu ( A ( \Omega_{t,l}  \wedge \Omega_{h,k} )) \onq,
 \end{eqnarray*}
are tangent to $\odr(q_0)$. This completes the proof because we have that $\nu (AB) \onq \in T_q \odr(q_0)$ if and only if $B \in \mathfrak{so} (T_x \G)$ for $\q \in \odr(q_0)$.
\end{proof}

We return to prove the main theorem in the beginning of the current subsection.

\begin{proof}[Proof of Theorem \ref{t3.2}].
As the vertical bundle of $Q$ is included in the tangent space of $\odr (q_0)$ by Proposition \ref{p3.7}, then the rolling problem $(\Sigma)_R$ is completely controllable (see Corollary 5.21 in \cite{ChitourKokkonen}). According to Theorem 4.3 in \cite{ChitourKokkonen2}, the holonomy group of $\G$ is equal to $\SO(m+n)$. Note that one could have used as well the main result in \cite{Ozeki} stating that the tangent space of the holonomy group at every point $x\in M$ contains the evaluations at $x$ of the curvature tensor and its covariant derivatives at any order.
\end{proof}

\subsection{Horizontal Holonomy Group of $(\G,g)$}

We define the distribution $\Delta : = \hbox{span}\{X_1 , \dots , X_m\}$ on $\G$ and $\qz \in Q$. Note that is of cnstant rank $m$. We will first compute a basis of $T_q \oDr (q_0)$ for any $q \in \oDr (q_0)$ and then determine the holonomy group $\mathcal{H}^{\;\nabla}_{\Dr}$ of rolling of $(\G,g)$ against $(\R^N, s_N)$, where $s_N$ is the Euclidean metric on $\R^N$.
\subsubsection{The Tangent Space of $\Oh_{\Delta_R}({q_0})$}

\begin{proposition}\label{p3.5}
For any $q_0\in Q$, the tangent space of $\Oh_{\Delta_R}({q_0})$ is generated by the following linearly independent vector fields:
\begin{align}\label{e3.22}
\lns (X_h) \onq, \; & \lns (A X_h) \onq, \; \lns(A \Omega_{h,k}) \onq, \; \lns (\Omega_{h,k}) \onq + \frac{1}{2} \nu ( A (X_h \wedge X_k)) \onq,\\
\nonumber & \nu \big( A (\sum_{j=1}^{m} X_j  \wedge \Omega_{h,j} ) \big) \onq, \; \nu \big( A (X_h \wedge X_k + \sum_{j=1}^{m} \Omega_{h,j}  \wedge \Omega_{k,j} ) \big) \onq.
\end{align}
\end{proposition}

\begin{proof}
Recall that all the data of the problem are analytic. Then, by the analytic version of the orbit theorem of Nagano-Sussmann (cf. \cite{AgrachevSachkov1}), the orbit $\Oh_{\Delta_R}({q_0})$ is an immersed analytic submanifold in the state space $Q$ and $T_q \Oh_{\Delta_R}({q_0}) = Lie_q(\Delta_R)$. Therefore, we are left to determine vector fields spanning $Lie_q(\Delta_R)$, i.e., to compute enough iterated Lie brackets of $\Delta$.

For any $h,k,l,s,t,p \in \{1,\dots,m\}$, we have
\begin{eqnarray}\label{e3.27}
\leftline{
$
\begin{array}{lcl}
& & [\lr(X_h),\lr(X_k)] \onq = \lr(\Omega_{h,k}) \onq + \nu ( A R (X_h, X_k)) \onq\\
\\
& = & \lr(\Omega_{h,k}) \onq + \frac{3}{4} \nu ( A (X_h \wedge X_k)) \onq + \frac{1}{4} \nu \big( A (\sum_{j=1}^{m} \Omega_{h,j}  \wedge \Omega_{k,j} ) \big) \onq.
\end{array}
$
}
\end{eqnarray}
\newline
\begin{eqnarray}\label{e3.28}
\leftline{
$
\begin{array}{lcl}
& & [\lr (X_l),[\lr(X_h),\lr(X_k)]] \onq\\
\\
& = & \delta_{kl} \; \big( \frac{3}{4} \lns(A X_h) \onq + \frac{1}{8} \nu ( A \sum_{j=1}^{m} X_j  \wedge \Omega_{h,j} ) ) \onq \big)\\
\\
& - & \delta_{hl} \; \big( \frac{3}{4} \lns(A X_k) \onq + \frac{1}{8} \nu ( A (\sum_{j=1}^{m} X_j  \wedge \Omega_{k,j} ) ) \onq \big).
\end{array}
$
}
\end{eqnarray}
\newline
\begin{eqnarray}\label{e3.29}
\leftline{
$
\begin{array}{lcl}
& & [ \lr (X_t) , [\lr (X_l),[\lr(X_h),\lr(X_k)]]] \onq\\
\\
& = & \delta_{kl} \; \big( \frac{1}{2} \lns(A \Omega_{t,h}) \onq +  \frac{1}{16} \nu ( A (X_t \wedge X_h)) \onq + \frac{1}{16} \nu ( A (\sum_{j=1}^{m} \Omega_{t,j}  \wedge \Omega_{h,j} ) ) \onq \big)\\
\\
& - & \delta_{hl} \; \big( \frac{1}{2} \lns(A \Omega_{t,k}) \onq +  \frac{1}{16} \nu ( A (X_t \wedge X_k)) \onq + \frac{1}{16} \nu ( A (\sum_{j=1}^{m} \Omega_{t,j}  \wedge \Omega_{k,j} ) ) \onq \big).
\end{array}
$
}
\end{eqnarray}
\newline
\begin{eqnarray}\label{e3.30}
\leftline{
$
\begin{array}{lcl}
& & [ \lr (X_s ) , [\lr (X_t) , [\lr (X_l),[\lr(X_h),\lr(X_k)]]]] \onq\\
\\
& = & \delta_{kl} \; \bigg( \delta_{hs} \big( \frac{3}{8} \lns(A X_t) \onq +  \frac{1}{32}  \nu ( A (\sum_{j=1}^{m} X_j  \wedge \Omega_{t,j} ) ) \big)\\
& & \quad \quad  - \delta_{ts} \big( \frac{3}{8} \lns(A X_h) \onq +  \frac{1}{32}  \nu ( A (\sum_{j=1}^{m} X_j  \wedge \Omega_{h,j} ) ) \big) \bigg)\\
\\
& - & \delta_{hl} \; \bigg( \delta_{ks} \big( \frac{3}{8} \lns(A X_t) \onq +  \frac{1}{32}  \nu ( A (\sum_{j=1}^{m} X_j  \wedge \Omega_{t,j} ) ) \big)\\
& & \quad \quad - \delta_{ts} \big( \frac{3}{8} \lns(A X_k) \onq +  \frac{1}{32}  \nu ( A (\sum_{j=1}^{m} X_j  \wedge \Omega_{k,j} ) ) \big) \bigg).
\end{array}
$
}
\end{eqnarray}
\newline
\begin{align}\label{e3.31}
\leftline{
$
\begin{array}{lcl}
& & [ \lr (X_p) , [\lr (X_s) , [\lr (X_t) , [\lr (X_l),[\lr(X_h),\lr(X_k)]]]]] \onq\\
\\
& = & \delta_{kl} \; \bigg( \delta_{hs} \big( \frac{7}{32} \lns(A \Omega_{p,t}) \onq +  \frac{1}{64} \nu ( A (X_p \wedge X_t)) \onq + \frac{1}{64} \nu ( A (\sum_{j=1}^{m} \Omega_{p,j}  \wedge \Omega_{t,j} ) ) \onq \big)\\
& & \quad \quad - \delta_{ts} \big( \frac{7}{32} \lns(A \Omega_{p,h}) \onq +  \frac{1}{64} \nu ( A (X_p \wedge X_h)) \onq + \frac{1}{64} \nu ( A (\sum_{j=1}^{m} \Omega_{p,j}  \wedge \Omega_{h,j} ) ) \onq \big) \bigg)\\
\\
& - & \delta_{hl} \; \bigg( \delta_{ks} \big( \frac{7}{32} \lns(A \Omega_{p,t}) \onq +  \frac{1}{64} \nu ( A (X_p \wedge X_t)) \onq + \frac{1}{64} \nu ( A (\sum_{j=1}^{m} \Omega_{p,j}  \wedge \Omega_{t,j} ) ) \onq  \big)\\
& & \quad \quad - \delta_{ts} \big( \frac{7}{32} \lns(A \Omega_{p,k}) \onq +  \frac{1}{64} \nu ( A (X_p \wedge X_k)) \onq + \frac{1}{64} \nu ( A (\sum_{j=1}^{m} \Omega_{p,j}  \wedge \Omega_{k,j} ) ) \onq \big) \bigg).
\end{array}
$
}
\end{align}
First we should remark that, by iteration, the commutators
\[
[\lr(X_{\alpha_r}), \dots [ \lr(X_{\alpha_2}) , \lr(X_{\alpha_1})]\dots],
\]
where $\alpha_i \in \{1,\dots,m\}$ and $r\geq 3$ are written either as the vectors in \eqref{e3.28} and \eqref{e3.30}, or as those in \eqref{e3.29} and \eqref{e3.31}. Therefore, one only has to prove that \eqref{e3.27}, \eqref{e3.28}, \eqref{e3.29}, \eqref{e3.30} and \eqref{e3.31} are linear combinations of the vector fields in \eqref{e3.22} and then show that the Lie algebra generated by these vector fields is involutive.
Indeed, fix some $h, k \in \{1,\dots,m\}$ such that $k \neq h$ and take $p=s=t=l=k$ in the above Lie brackets. Calculate $\frac{1}{4} \eqref{e3.28} + \eqref{e3.30}$ and $\frac{1}{2} \eqref{e3.28} + \eqref{e3.30}$, we get that
$$
\lns (A X_h) \onq, \quad \nu ( A (\sum_{j=1}^{m} X_j  \wedge \Omega_{h,j} ) )\onq
$$
are vectors in $Lie(\Delta_R)$. On the other hand, $\frac{1}{4} \eqref{e3.29} + \eqref{e3.31}$ and $\frac{7}{16} \eqref{e3.29} + \eqref{e3.31}$ imply that
$$
\lns(A \Omega_{h,k}) \onq, \quad \nu ( A (X_h \wedge X_k + \sum_{j=1}^{m} \Omega_{h,j}  \wedge \Omega_{k,j} ) ) \onq
$$
belong also to $Lie(\Delta_R)$. The last two vectors with $\eqref{e3.27}$ give us another vector in $Lie(\Delta_R)$ which is
$$
\lns (\Omega_{h,k}) \onq + \frac{1}{2} \nu ( A (X_h \wedge X_k)) \onq.
$$
To show that the vector fields given in Eq.~\eqref{e3.22} form a basis for $Lie(\Delta_R)$, it remains to compute their first order Lie brackets to see that they define an involutive Lie algebra.

We have,
\begin{equation}\label{lie1}
\leftline{
$
\begin{array}{ll}
& [ \nu ( (\cdot) (\sum_{j=1}^{m} X_j  \wedge \Omega_{h,j} ) ),\nu ( (\cdot) (\sum_{j=1}^{m} X_j \wedge \Omega_{k,j} ) )] \onq \\
\\
= & \nu ( A (X_h \wedge X_k + \sum_{j=1}^{m} \Omega_{h,j}  \wedge \Omega_{k,j} )) \onq,\\
\end{array}
$
}
\end{equation}
\newline
\begin{equation}\label{lie2}
\leftline{
$
\begin{array}{ll}
& [ \nu ( (\cdot) (\sum_{j=1}^{m} X_j  \wedge \Omega_{l,j} ) ), \nu ( (\cdot) (X_h \wedge X_k + \sum_{j=1}^{m} \Omega_{h,j}  \wedge \Omega_{k,j} )) ] \onq\\
\\
= & \delta_{kl} \nu ( A (\sum_{j=1}^{m} X_j  \wedge \Omega_{h,j} ) ) \onq - \delta_{hl} \nu ( A (\sum_{j=1}^{m} X_j  \wedge \Omega_{k,j} ) ) \onq,
\end{array}
$
}
\end{equation}
and,
\begin{equation}\label{lie3}
\leftline{
$
\begin{array}{ll}
& [ \nu ( (\cdot) (X_l \wedge X_t+ \sum_{j=1}^{m} \Omega_{l,j}  \wedge \Omega_{t,j} )) , \nu ( (\cdot) (X_h \wedge X_k + \sum_{j=1}^{m} \Omega_{h,j}  \wedge \Omega_{k,j} ))] \onq\\
\\
=  & \delta_{kl} \nu ( A (X_h \wedge X_t + \sum_{j=1}^{m} \Omega_{h,j}  \wedge \Omega_{t,j} )) \onq +  \delta_{hl} \nu ( A (X_t \wedge X_k + \sum_{j=1}^{m} \Omega_{t,j}  \wedge \Omega_{k,j} )) \onq\\
\\
+ &  \delta_{tk} \nu ( A (X_l \wedge X_h + \sum_{j=1}^{m} \Omega_{l,j}  \wedge \Omega_{h,j} )) \onq + \delta_{ht} \nu ( A (X_k \wedge X_l + \sum_{j=1}^{m} \Omega_{k,j}  \wedge \Omega_{l,j} )) \onq.
\end{array}
$
}
\end{equation}
Moreover, the Lie brackets between $\lns (\Omega_{h,k}) + \frac{1}{2} \nu ( A (X_h \wedge X_k))$ and the remaining vectors of \eqref{e3.22} are
\begin{equation*}
\begin{split}
&[\lns (X_l) , \lns (\Omega_{h,k}) + \frac{1}{2} \nu ( (\cdot) (X_h \wedge X_k)) ] \onq = 0,\\
&[\lns ((\cdot ) X_l) , \lns (\Omega_{h,k}) + \frac{1}{2} \nu ( (\cdot) (X_h \wedge X_k)) ] \onq = 0,\\
&[\lr ((\cdot ) \Omega_{i,j}) , \lns (\Omega_{h,k}) + \frac{1}{2} \nu ( (\cdot) (X_h \wedge X_k)) ] \onq = 0,\\
&[\lns ( (\cdot ) \Omega_{l,t}) , \lns (\Omega_{h,k}) + \frac{1}{2} \nu ( (\cdot) (X_h \wedge X_k)) ] \onq = 0,\\
&[\lns (\Omega_{h,k}) + \frac{1}{2} \nu ( (\cdot) (X_h \wedge X_k)) , \lns (\Omega_{l,t}) + \frac{1}{2} \nu ( (\cdot) (X_l \wedge X_t)) ] \onq = 0,\\
&[\lns (\Omega_{h,k}) + \frac{1}{2} \nu ( (\cdot) (X_h \wedge X_k)), \nu ( (\cdot) (X_l \wedge X_t + \sum_{j=1}^{m} \Omega_{l,j}  \wedge \Omega_{t,j} ))] \onq = 0,\\
&[\lns (\Omega_{h,k}) + \frac{1}{2} \nu ( (\cdot) (X_h \wedge X_k)), \nu ( (\cdot) (\sum_{j=1}^{m} X_j  \wedge \Omega_{l,j} ) )] \onq = 0.
\end{split}
\end{equation*}
Then, the vector fields
\begin{align*}
\lns (X_h) \onq, \; & \lns (A X_h) \onq, \; \lns(A \Omega_{h,k}) \onq, \; \lns (\Omega_{h,k}) \onq + \frac{1}{2} \nu ( A (X_h \wedge X_k)) \onq,\\
& \nu \big( A (\sum_{j=1}^{m} X_j  \wedge \Omega_{h,j} ) \big) \onq, \; \nu \big( A (X_h \wedge X_k + \sum_{j=1}^{m} \Omega_{h,j}  \wedge \Omega_{k,j} ) \big) \onq,
\end{align*}
form an involutive distribution and any vector fields in $Lie(\Delta_R)$ is a linear combination of them. It remains to check that they are linearly independent. It is clearly enough to do that for the family of vector fields $\nu \big( A (\sum_{j=1}^{m} X_{j}  \wedge \Omega_{h,j} ) \big) \onq$. Suppose there exists $(\alpha_h)_{1\leq h\leq m}$, such that $\sum_{h=1}^{m}{\alpha_h}\sum_{j=1}^{m} X_j \wedge \Omega_{h,j} =0 $. Then $\sum_{j=1}^{m} X_j \wedge (\sum_{h=1, h\neq j}^{m} {\alpha_h} \Omega_{h,j} )=0$ $\forall j \in\{1,\dots,m\}$. Hence, $\sum_{h=1, h\neq j}^{m} {\alpha_h} \Omega_{h,j} =0$ for every $j,h \in\{1,\dots,m\}$, so ${\alpha_h}=0$ for every $h \in\{1,\dots,m\}$. Therefore,
\begin{align*}
\lns (X_h) \onq, \; & \lns (A X_h) \onq, \; \lns(A \Omega_{h,k}) \onq, \; \lns (\Omega_{h,k}) \onq + \frac{1}{2} \nu ( A (X_h \wedge X_k)) \onq,\\
& \nu \big( A (\sum_{j=1}^{m} X_j  \wedge \Omega_{h,j} ) \big) \onq, \; \nu \big( A (X_h \wedge X_k + \sum_{j=1}^{m} \Omega_{h,j}  \wedge \Omega_{k,j} ) \big) \onq,
\end{align*}
is a global basis of $Lie_q(\Delta_R)$ and hence the dimension of $Lie_q(\Delta_R)$ is constant and equal to $3 N$. We deduce that $\dim \Oh_{\Delta_R} (q_0) = 3N$ and the tangent space of $\Oh_{\Delta_R}({q_0})$ is generated by the vectors in \eqref{e3.22}.
\end{proof}

\begin{remark}\label{r3.2}
According to this proposition, $\lns (A X_h) \onq$ and $\lns(A \Omega_{h,k}) \onq$ are tangent to $\Oh_{\Delta_R} (q_0)$. This implies that $\pi_{Q,\R^N}(\Oh_{\Delta_R}) = \R^{m+n}$ which means that all the translations along $\R^{m+n}$ are included in the tangent space of the orbit. Furthermore, the families of vector fields $\nu \big( A (\sum_{j=1}^{m} X_j  \wedge \Omega_{h,j} ) \big) \onq $ and $\nu \big( A (X_h \wedge X_k + \sum_{j=1}^{m} \Omega_{h,j}  \wedge \Omega_{k,j} ) \big) \onq $ form an involutive vertical distribution.
\end{remark}

\subsubsection{Determination of $\mathcal{H}^{\;\nabla}_{\Dr |q}$}

The main result of the subsection is given next.

\begin{proposition}\label{p3.4}
\begin{itemize}
\item[(i)]
The affine $\Delta$-horizontal holonomy group $\mathcal{H}^{\;\nabla}_{\Dr}$ is a Lie subgroup of $\SE(m+n)$ of dimension  $2(m+n)$.
\item[(ii)] The $\Delta$-horizontal holonomy group $H_\Delta^{\; \nabla}$  is a Lie subgroup of $\SO(m+n)$ of dimension $m+n$.
Moreover, the connected component of the identity $(H_\Delta^{\; \nabla})_0$ of
$H_\Delta^{\; \nabla}$ is compact.
\end{itemize}
\end{proposition}

\begin{proof} 
As an immediate adaptation of Proposition~\ref{p3.10} to the case where one is dealing with principal $\SE(n)$-bundles, one gets that the affine holonomy group $\mathcal{H}^{\;\nabla}_{\Dr}$ is a Lie subgroup of $\SE(m+n)$. Notice then that, if $\Pi:\SE(m+n)\to \SO(m+n)$ is the projection onto the $\SO(m+n)$ factor of $\SE(m+n)$, one has, by definition 
$H^\nabla = \Pi(\mathcal{H}^\nabla)$ and 
$H^{\;\nabla}_{\Delta} = \Pi(\mathcal{H}^{\;\nabla}_{\Delta_R})$. This shows that the $\Delta$-horizontal holonomy group $H_\Delta^{\; \nabla}$ is a Lie subgroup of $\SO(m+n)$.

We next prove that for every $q'\in Q$,
a basis of $\mathrm{Lie}(\mathcal{H}^{\;\nabla}_{\Dr |q'})$ the Lie-algebra of $\mathcal{H}^{\;\nabla}_{\Dr |q'}$ is given by the evaluation at $q'$ of
the vector fields whose values at $q=(x,\hx,A)\in Q$ are
\begin{align}\label{goodH}
\lns (A X_h) \onq, \lns(A \Omega_{h,k}) \onq, \nu \big( A (\sum_{j=1}^{m} X_j  \wedge \Omega_{h,j} ) \big) \onq,\nu \big( A (X_h \wedge X_k + \sum_{j=1}^{m} \Omega_{h,j}  \wedge \Omega_{k,j} ) \big) \onq,
\end{align}
and hence this Lie algebra has dimension $2(m+n)=m(m+1)$. To see that, consider some element $V\in \mathrm{Lie}(\mathcal{H}^{\;\nabla}_{\Dr |q'})$ as a linear subspace of $T_{q'} \Oh_{\Delta_R}({q_0})$. Then $V$ is a linear combination of the vector fields described in Eq.~\eqref{e3.22} evaluated at $q'$ and $V$
projects to a zero-vector in $TM$. By an obvious computation, one deduces that 
 $V$ is a linear combination of the vector fields given in Eq.~\ref{goodH}. Conversely, it is clear that 
 the vector fields given in Eq.~\eqref{goodH} generate a distribution whose integral manifolds lie in $\oDr (q_0) \cap \pi_{Q,M}^{-1} (x')$ where $x'=\pi_{Q,M}(q')$.
 This proves that $\mathrm{Lie}(\mathcal{H}^{\;\nabla}_{\Dr |q'})$ the Lie-algebra of $\mathcal{H}^{\;\nabla}_{\Dr |q'}$ has dimension $2(m+n)=m(m+1)$. One could also check that the distribution generated by the vector fields in Eq.~\eqref{goodH} is involutive. 
 
By a similar reasoning, a basis $\mathrm{Lie}(H_{\Delta |q'}^{\; \nabla})$ of the Lie-algebra of $H_{\Delta |q'}^{\; \nabla}$ is given by
the (evaluations at $q'$ of) vector fields
(see also \eqref{liealL} below)
\begin{align*}
\nu \big( A (\sum_{j=1}^{m} X_j  \wedge \Omega_{h,j} ) \big) \onq, \; \nu \big( A (X_h \wedge X_k + \sum_{j=1}^{m} \Omega_{h,j}  \wedge \Omega_{k,j} ) \big) \onq,
\end{align*}
and hence this Lie algebra has dimension $m+n=m(m+1)/2$.

It remains to prove the last claim in (ii).
For $1\leq h\leq m$, let $A_h\in \mathfrak{so}(m+n)$ corresponding to the vertical vector $\nu \big( A (\sum_{j=1}^{m} X_j  \wedge \Omega_{h,j} ) \big)$ and, for $(h,k)\in \I$, let $B_{h,k}\in\mathfrak{so}(m+n)$ corresponding to the vertical vector $\big( A (X_h \wedge X_k + \sum_{j=1}^{m} \Omega_{h,j}  \wedge \Omega_{k,j} ) \big)$. We extend the notations for the $B_{h,k}$ to
for any $1\leq h,k\leq m$ by setting $B_{h,k}=-B_{k,h}$.
The basis of Lie algebra $L:=\mathrm{Lie}(H^{\;\nabla}_{\Delta})$ of $H^{\;\nabla}_{\Delta}$
is given by the matrices $A_h$, $1\leq h\leq m$
and $B_{h,k}$, $(h,k)\in \I$ and thanks to Eqs.~\eqref{lie1}, \eqref{lie2} and \eqref{lie3},
one has
\begin{equation}\label{liealL}
[A_i,A_j]=B_{i,j},\ [A_i,B_{h,k}]=\delta_{ki}A_h-\delta_{hi}A_k,\ [B_{l,t},B_{h,k}]=\delta_{kl}B_{h,t}+
\delta_{hl}B_{t,k}+\delta_{tk}B_{l,h}+\delta_{ht}B_{k,l}.
\end{equation}
We prove in Subsection~\ref{semi0} that $L$ is a compact semisimple Lie algebra. 
Since the connected component of 
identity $\left(H_{\Delta}^{\; \nabla}\right)_0$  of $H_{\Delta}^{\; \nabla}$ is a connected Lie with Lie algebra $L$,
it follows from Weyl's theorem (cf. \cite[Theorem 26.1]{postnikov})
that $\left(H_{\Delta}^{\; \nabla}\right)_0$ is compact in
$\SO(m+n)$ and of dimension $m+n$.
\end{proof}

Since $m+n\leq (m+n)(m+n-1)/2$
(resp. $2(m+n)\leq (m+n+1)(m+n)/2$)
for all $m\geq 2$ and equality holds
if and only if $m=2$ we have obtain our last result.

\begin{corollary}
In the set up of Proposition \ref{p3.4},
the inclusions $\mathcal{H}^{\;\nabla}_{\Dr}\subset\SE(m+n)$
and $H_\Delta^{\; \nabla}\subset \SO(m+n)$
are strict if and only if $m\geq 3$.
\end{corollary}

\section{Appendix}
\subsection{\emph{o-regular} controls}
We first generalize the usual definition of regular control and then provide a 
result about existence of such controls. Let $M$ be an $n$-dimensional smooth manifold,
$\mc{F}$ a (possibly infinite) family of smooth vector fields on $M$, and let $\Delta_{\mc{F}}$ 
be the smooth \emph{singular} distribution (cf. \cite{harms12}) spanned by $\mc{F}$,
i.e.
\[
\Delta_{\mc{F}}|_p=\hbox{span}\{X|_p\ |\ X\in\mc{F}\}\subset T_p M,\quad p\in M.
\]
We use the word "singular" (to emphasize the fact that the rank (dimension) of $\Delta_{\mc{F}}$ might vary from point to point.
One can, in fact, prove given any such family $\mc{F}$,
there is a \emph{finite} subfamily $\mc{F}_0=\{X_1,\dots,X_m\}$ such that $\Delta_{\mc{F}}=\Delta_{\mc{F}_0}$,
and $m\leq n(n+1)$ (see \cite{drager12,sussmann08}, or \cite{Rifford} when $\Delta_{\mc{F}}$ has constant rank).
Moreover, by $\mathrm{span}\ S$ we mean $\R$-linear span of a set $S$.

\begin{definition}\label{def:orbit}
An absolutely continuous (a.c.) curve $\gamma:[0,T]\to M$ is horizontal with respect to $\mc{F}$
if there is a finite subfamily $\{X_1,\dots,X_m\}$ of $\mc{F}$
and $u=(u_1,\dots,u_d)\in L^1([0,T],\R^m)$, $m\in\mathbb{N}$
(here $m$ might depend on the curve $\gamma$ in question),
such that for almost every $t\in [0,T]$,
\[
\dot{\gamma}(t)=\sum_{i=1}^m u_i(t)X_i|_{\gamma(t)}.
\]
The \emph{orbit} $\mc{O}_{\mc{F}}(p)$ of $\mc{F}$ through $p\in M$
is the set of all points of $M$ reached by $\mc{F}$-horizontal paths $\gamma$
with $\gamma(0)=p$.
\end{definition}

If $\Delta$ is a smooth distribution of \emph{constant rank} $k$ on $M$,
and if $\mc{F}=\mc{F}_{\Delta}$ is the set of smooth vector fields tangent to $\Delta$,
then it is easy to see that $\Delta=\Delta_{\mc{F}}$,
and that an a.c. curve is $\Delta$-horizontal 
if and only if it is $\mc{F}$-horizontal.
Therefore, in this case the concept of orbit coincides with the notion
we have used previously in the paper,
and one can without ambiguity denote it by $\mc{O}_{\Delta}(p)$ instead of $\mc{O}_{\mc{F}}(p)$.

For a smooth vector field $X$ write $\Phi_X:D\to M$ for its flow,
where $D=D_X$ is an open connected subset of $\R\times M$ containing $\{0\}\times M$.
We also use the notation $(\Phi_X)_t(x)=(\Phi_X)^x(t)=\Phi_X(t,x)$ when $(x,t)\in D$.

The orbit of a family $\mc{F}$ of vector fields has the following properties (cf. \cite{harms12}, \cite{Jean}).

\begin{theorem}[Orbit Theorem]\label{th:orbit}
\begin{enumerate}
\item The orbit $\Oh_{\mc{F}}(p)$ is an immersed submanifold of $M$.
\item Any continuous (resp. smooth) map $f:Z\to M$, where $Z$ is a smooth manifold,
such that $f(Z)\subset \Oh_{\mc{F}}(p)$ is continuous (resp. smooth) as a map $f:Z\to \Oh_{\mc{F}}(p)$.
\item If one writes $G_{\mc{F}}$ for the set of all locally defined diffeomorphisms of $M$
of the form $(\Phi_{X_r})^{t_1}\circ\dots\circ (\Phi_{X_d})^{t_d}$
for $X_1,\dots,X_d\in \mc{F}$ and $t_1,\dots,t_d\in\R$ for which this map is defined,
then
\[
\Oh_{\mc{F}}(p)&=\{\varphi(p)\ |\ \varphi\in G_{\mc{F}}\} \\
T\Oh_{\mc{F}}(p)&=\mathrm{span}\{\varphi_*(X)\ |\ \varphi\in G_{\mc{F}},\ X\in\mc{F}\},
\]
wherever the expressions $\varphi(p)$ and $\varphi_*(X)$ are defined.
\end{enumerate}
\end{theorem}

As a consequence of Case 3. of the theorem,
one sees that $L^1([0,T],\R^m)$ in Definition \ref{def:orbit}
can be replaced by $L^2([0,T],\R^m)$, which for the rest paper
will be the appropriate \emph{space of controls} for our needs.

Following \cite{Rifford} we define the concepts of the end-point mapping and that
of a regular ($L^2$-)control.

\begin{definition}\label{end-point}
For every $p\in M$, any time $T>0$,
and any smooth finite family of vector fields $\mc{F}=\{X_1,\dots,X_m\}$ on $M$, there exists a maximal
open subset $U^{p,T}_{\mc{F}}\subset L^2([0,T],\mathbb{R}^m)$ such that for every $u=(u_1,\dots,u_m)\in U_p^T$, there exists a unique
absolutely continuous solution $\gamma_u:[0,T]\to M$ to the Cauchy problem
\begin{align}\label{eq:CE}
\dot{\gamma_u}(t)=\sum_{i=1}^m u_i(t)X_i(\gamma_u(t)),\quad \gamma_u(0)=p.
\end{align}

The end-point map $E^{p,T}_{\mc{F}}$ associated to $\mc{F}$ at $p$ in time $T$ is defined as the mapping
\[
E^{p,T}_{\mc{F}}:U^{p,T}_{\mc{F}}\rightarrow M,\quad E^{p,T}_{\mc{F}}(u)=\gamma_u(T).
\]
\end{definition}

By \cite[Proposition 1.8]{Rifford} we have the following.

\begin{proposition}
With $p,T,\mc{F}$ as above, the end point map $E^{p,T}_{\mc{F}}:U^{p,T}_{\mc{F}}\rightarrow M$
is $C^1$-smooth.
\end{proposition}

This proposition allows us to give the following definition.

\begin{definition}\label{def:regcontrol}
A control $u\in U^{p,T}_{\mc{F}}$ is said to be \emph{o-regular} with respect to $p$ in 
time $T$ if the rank of $D_u E_{\mc{F}}^{p,T}:L^2([0,T],\R^m)\to T_{E_{\mc{F}}^{p,T}(u)} M$, the differential of $E_{\mc{F}}^{p,T}(u)$ at $u$,
is equal to $\dim \Oh_{\mc{F}}(p)$. Here, ''o-regular'' stands for orbitally regular. 
\end{definition}

\begin{remark}
A control $u$ is usually said to be regular (with respect to $p$ in time $T$) if  the rank of $E_{\mc{F}}^{p,T}(u)$ is equal to the dimension $n$ of the ambient manifold $M$ (cf \cite[Section 1.3]{Rifford}), implying in particular that the orbit $\Oh_{\mc{F}}(p)$ is open in $M$ and thus is $n$-dimensional. If the distribution generated by $\mc{F}$ verifies the LARC, it can be proved that any pair of points in $M$ can be joined by the trajectory tangent to this distribution and corresponding to a regular control, cf. \cite{Bellaiche}. In this paper, we have extended this definition without assuming controllability. 
\end{remark}

The main purpose of this appendix is to generalize the result of 
\cite{Bellaiche} 
to the case where the distribution $\Delta$ is not necessarily bracket-generating. 
Indeed, we have the following result.

\begin{proposition}\label{regular-0}
Let $M$ be an $n$-dimensional smooth manifold, $\mc{F}=\{X_1,\dots,X_m\}$, $m\in\mathbb{N}$,
a smooth \emph{finite} family of vector fields on $M$.
Then, for every $p\in M$ and time $T>0$, and every $q\in \Oh_{\mc{F}}(p)$, there exists
a \emph{o-regular} control with respect to $p$ in time $T$ such that the unique solution $\gamma_u$ to the Cauchy problem \eqref{eq:CE} such that $\gamma_u(T)=q$.
\end{proposition}

\begin{remark}\label{re:regdense}
By the proof of Proposition 1.12 in \cite{Rifford} (see also \cite{harms12}), the conclusion is immediate if $T_q\Oh_{\mc{F}}(p)$ is equal for every $q\in\Oh_{\mc{F}}(p)$ to $\mathrm{Lie}_q(\cal{F})$, the evaluation at $q$ of the Lie algebra generated by $\cal{F}$.
In fact, in this case a stronger result holds, namely the set of regular controls is dense
in $U^{q,T}_{\mc{F}}$ for every $q\in \mc{O}_{\mc{F}}(p)$ and $T>0$.
As a consequence, any control $u_0\in E^{p,T}_{\mc{F}}$
admits an o-regular control $u$ arbitrarily close (in $L^2$) to $u_0$
such that $E^{p,T}_{\mc{F}}(u)=E^{p,T}_{\mc{F}}(u_0)$.
\end{remark}

\begin{proof}
Fix $q_0\in \Oh_{\mc{F}}(p)$ and $(Z^0_1,\dots,Z^0_d)$ a basis of $T_{q_0}\Oh_{\mc{F}}(p)$.
According to Theorem \ref{th:orbit}, there exists $\varphi_1\in G_{\mc{F}}$ and $Y_1\in \mc{F}$ with $q_1:=\varphi_1^{-1}(q_0)$
such that $Z^0:=(\tilde{Z}^0_1,Z^0_2,\dots,Z^0_d)$,
where $\tilde{Z}^0_1=(\varphi_{1})_*Y_1|_{q_0}$,
forms a basis of $\Oh_{\mc{F}}(p)$ at $q_0$.
The basis $Z^0$ is the pushforward of a basis $Z^1=(Z^1_1,\dots,Z^1_d)$ of $T_{q_1}\Oh_{\mc{F}}(p)$ by $\varphi_1$ and obviously $Z^1_1=Y_1$.
We proceed inductively (using Theorem \ref{th:orbit}) with this construction for $1\leq l\leq d$ so that the basis $Z^{l-1}=(Z^{l-1}_1,\dots,Z^{l-1}_{l-1},\tilde{Z}^{l-1}_{l},Z^{l-1}_{l+1},\dots,Z^{l-1}_d)$ of 
$T_{q_{l-1}}\Oh_{\mc{F}}(p)$ is the pushforward of a basis $Z^l=(Z^l_1,\dots,Z^l_d)$ with $q_l:=\varphi_l^{-1}(q_{l-1})$, $Y_l:=Z^l_l\in \mc{F}$ and $\tilde{Z}^{l-1}_l=(\varphi_l)_*(Z^l_l)$.
Finally consider $\varphi_{d+1}\in G$ so that $\varphi_{d+1}(p)=q_n$ and set $\psi=\varphi_1\circ\varphi_2\circ\cdots\circ\varphi_{d+1}$.
One has that $\psi(p)=q_0$ and there exists $T>0$ and $u\in L^2([0,T],\mathbb{R}^m)$ such that the unique solution $\gamma_u$ to the Cauchy problem $\dot{\gamma}_u(t)=\sum_{i=1}^m u_i(t)X_i|_{\gamma_u(t)}$,  $x(0)=p$
verifies $\gamma_u(T)=q_0$. Then the flow of diffeomorphisms $\psi^{u}(t,q)$ corresponding to the time-varying vector field $q\mapsto \sum_{i=1}^m u_i(t)X_i|_q$ verifies $\psi=\psi^{u}(T,0)$ and $\psi^{u}(t,p)=\gamma_u(t)$ where one has, for $0\leq s\leq t\leq T$,  $\frac{\partial\psi^{u}(t,q)}{\partial t}=\sum_{i=1}^m u_i(t)X_i|_{\psi^{u}(t,q)}$ together with the initial condition $\psi^{u}(0,q)=q$ for every $q\in M$. With the above notations, it is clear that, for every $1\leq l\leq d$, 
\[
\big(d_q\psi^{u}(T,p)(d_q\psi^{u}(t_l,p))^{-1}Y_l\big)_{l=1,\dots,d}=:(\tilde{Z}_1,\dots,\tilde{Z}_d)
\]
forms a basis of $T_{q_0}\Oh_{\mc{F}}(p)$,
where $d_q\psi^{u}(t,\cdot)$ denotes the differential of $\psi(t,q)$ with respect to the $q$ variable.

Recall that the differential of the end-point map at $u$ is the linear map 
$D_u E^{p,T}_{\mc{F}}: L^2([0,T],\mathbb{R}^m)\rightarrow T_{q_0}\Oh_{\mc{F}}(p)$ given by 
\begin{equation}\label{dE0}
D_u E^{p,T}_{\mc{F}}(v)=d_q\psi^{u}(T,p)\int_0^T(d_q\psi^{u}(t,p))^{-1}X^v(t,\gamma_u(t))dt,
\end{equation}
where $X^v(t,x)=\sum_{i=1}^m v_i(t)X_i|_{x}$ for almost every $t\in [0,T]$ and every $x\in M$. 
We further complete the notations as follows. Let $0=t_0<t_1<\cdots<t_{d+1}:=T$ the sequence of times where $\gamma_u(t_l)=q_{d+1-l}$ with the convention that $p=q_{d+1}$ and thus $\psi^{u}(T,p)\psi^{u}(t_l,p)^{-1}(q_l)=q_0$, for $0\leq l\leq d+1$. Moreover, one has 
$Y_l=\sum_{i=1}^m y_{il}X_i|_{q_l}$ for $1\leq l\leq d$ and some real numbers $(y_{il})$. 

For every $\varepsilon>0$  small enough and $1\leq l\leq d$, consider the sequence $(v^l_{\varepsilon})$ of functions in $L^2([0,T],\mathbb{R}^m)$ defined by $v^l_{\varepsilon}(t)=\frac1{\varepsilon}{(y_{il})_{1\leq i\leq k}}$ if $t_l-\varepsilon\leq t\leq t_l$ and zero otherwise. It is a matter of standard computations (as performed in \cite[Proposition 1.10]{Rifford} to prove that, for every $1\leq l\leq d$, $D_u E^{p,T}_{\mc{F}}(v^l_{\varepsilon})$ tends to $d_q\psi^{u}(T,p)(d_q\psi^{u}(t_l,p))^{-1}Y_l=\tilde{Z}_l$ as $\varepsilon$ tends to zero. Since the the range of $D_u E^{p,T}_{\mc{F}}$ is closed, we deduce that it contains $\tilde{Z}_l$ for every $1\leq l\leq d$. 

We have therefore proved that $u$ is o-regular at $p$ in time $T$ in the sense of Definition~\ref{end-point}.
\end{proof}

\begin{remark}
In contrast to what was discussed in Remark \ref{re:regdense},
we highlight the fact that
in general case where the (finite) family $\mc{F}$
of vector fields does not satisfy (everywhere on the orbit)
the H\"ormander condition $\mathrm{Lie}_q\mc{F}=T_q\mc{O}_{\mc{F}}(p)$,
for a given control $u_0\in U^{p,T}_{\mc{F}}$ the o-regular controls $u$
(in the sense of Definition \ref{def:regcontrol})
such that $E^{p,T}_{\mc{F}}(u_0)=E^{p,T}_{\mc{F}}(u)$
might lie far away from $u_0$ in $L^2$-sense.

As the standard example,
consider on $M=\R^2$, with coordinates $(x,y)$, the vector fields (cf. \cite{harms12}, p.12)
$X=\pa{x}$ and $Y=\phi(x)\pa{y}$
where $\phi:\R\to\R$ is smooth such that $\phi(x)=0$ if $x\leq 0$
and $\phi(x)>0$ for $x>0$. Let $\mc{F}=\{X,Y\}$ and $\R^2_-=\{(x,y)\ |\ x<0\}$

It is clear that for any point $p_0=(x_0,y_0)$ with $x_0<0$, any $T>0$ and any control $u_0$ such that $E_{\mc{F}}^{p_0,t}(u_0)\in \R^2_-$ for all $t\in [0,T]$,
there is an $L^2$-neighbourhood of $u_0$
such that $E_{\mc{F}}^{p_0,T}$ is not regular at any of its points.

A regular control $u$ steering $p_0$ to $q_0=E_{\mc{F}}^{p_0,T}(u_0)$ in time $T$
(i.e. $E_{\mc{F}}^{p_0,T}(u)=q$), which exists thanks to Proposition \ref{regular-0},
must have the property that $E_{\mc{F}}^{p_0,T_0}(u)\notin\R^2_-$
for some $0<T_0\leq T$.
Therefore, if we write $\gamma_u(t)=(x_u(t),y_u(t))=E^{p_0,t}(u)$
and $u=(u_1,u_2)$, one has
\[
|x_0|\leq |x_u(T_0)-x_0|=\big|\int_0^{T_0} u_1(s)ds\big|\leq \sqrt{T_0}\n{u_1}_{L^2([0,T])}
\leq \sqrt{T}\n{u}_{L^2([0,T],\R^2)}.
\]
If for example one took $u_0=0$, hence $q_0=p_0$,
the above inequality would prove, as was claimed above,
that a regular control $u$ steering $p_0$ to $q_0$ in time $T$
cannot be near $u_0$ in $L^2$-sense.
\end{remark}

\subsection{Semisimplicity of the Lie algebra $L$}\label{semi0}

In this paragraph, we prove that the Lie algebra $L$ of $H_\Delta^{\; \nabla}$ whose generators are given in Eq.~\eqref{liealL} is compact semisimple.
In order to so, according to the proof of Proposition 26.3 in \cite{postnikov},
it is enough (and necessary) to show that $L$ is compact and has trivial center.
We also recall that a Lie algebra $\mathfrak{g}$
is called compact (\cite[Definition 26.2]{postnikov})
if there is a positive definite inner product $k$ on $\mathfrak{g}$
which satisfies
\begin{align}\label{eq:compact}
k([x,y],z)+k(y,[x,z])=0,\quad \forall x,y,z\in\mathfrak{g}.
\end{align}
It follows immediately that any Lie-subalgebra $\mathfrak{h}$ of a compact Lie-algebra $\mathfrak{g}$
is also compact, and therefore, $L$ as a Lie-subalgebra of the compact $\mathfrak{so}(n+m)$ is compact.

It remains to show that the center $L$ is trivial.
Consider $C$ in the center of $L$, i.e., $[C,X]=0$ for every $X\in L$. Let $\mathfrak{a}$ and $\mathfrak{b}$ be respectively the $\R$-linear span of the $(A_h)_{1\leq h\leq m}$ and the span of the $(B_{h,k})_{(h,k)\in \I}$. Note that $\mathfrak{b}=\mathfrak{so}(m)$ (up to an isomorphism of Lie algebras) and $L$ is the direct sum of $\mathfrak{a}$ and $\mathfrak{b}$.
Thus we can write $C=A+B$ with unique $A\in \mathfrak{a}$ and $B\in \mathfrak{b}$.

For every $(h,k)\in \I$, one has $0=[C,B_{h,k}]=[A,B_{h,k}]+[B,B_{h,k}]$. 
Thanks to the relations in Eq.~\eqref{liealL}, one also has that $[A,B_{h,k}]\in \mathfrak{a}$
and $[B,B_{h,k}]\in \mathfrak{b}$, and then, due to the direct sum property one concludes that, for every $(h,k)\in \I$,
$$
[A,B_{h,k}]=[B,B_{h,k}]=0.
$$
Since $\mathfrak{b}$ is semisimple, its center reduces to zero and thus $B=0$. We next 
set $A=\sum_{l=1}^ma_lA_l$ and use the relations $[A,A_h]=0$ for $1\leq h\leq m$. We get that, for $1\leq h\leq m$,
$$
0=\sum_{l=1}^ma_l[A_l,A_h]=\sum_{l=1}^ma_l B_{l,h},
$$
yielding at once that $a_l=0$ for $1\leq h\leq m$ because $m\geq 2$. (Indeed we need at least two distinct indices $h$ as above.) Then $C=0$ which concludes the proof of the claim.



\end{document}